\documentclass[12pt, twoside]{article}
\usepackage{amsmath,amsthm,amssymb}
\usepackage{times}
\usepackage{enumerate}
\usepackage{extarrows}
\usepackage{color}
\usepackage{tabularx}
\usepackage{multirow}
\usepackage{graphicx}
\usepackage[numbers,sort&compress]{natbib}

\def\titlerunning#1{\gdef\titrun{#1}}
\makeatletter
\def\author#1{\gdef\autrun{\def\and{\unskip, }#1}\gdef\@author{#1}}
\def\address#1{{\def\and{\\\hspace*{18pt}}\renewcommand{\thefootnote}{}%
\footnote {#1}}%
\markboth{\autrun}{\titrun}}
\makeatother
\def\email#1{e-mail: #1}
\def\subjclass#1{{\renewcommand{\thefootnote}{}%
\footnote{\emph{Mathematics Subject Classification (2020):} #1}}}
\def\keywords#1{\par\medskip
\noindent\textbf{Keywords.} #1}

\newtheorem{theorem}{Theorem}
\newtheorem{lemma}{Lemma}[section]
\newtheorem{definition}[lemma]{Definition}
\newtheorem{thm}[lemma]{Theorem}
\newtheorem{proposition}[lemma]{Proposition}
\newtheorem{remark}[lemma]{Remark}
\newtheorem{corollary}[lemma]{Corollary}
\newtheorem{example}[lemma]{Example}

\newcommand{\gm}{\gamma}
\newcommand{\R}{\mathbb{R}}

\newcommand{\Proof}{\begin{proof}}
\newcommand{\End}{\end{proof}}

\numberwithin{equation}{section}

\newcommand{\PreserveBackslash}[1]{\let\temp=\\#1\let\\=\temp}
\newcolumntype{C}[1]{>{\PreserveBackslash\centering}p{#1}}
\newcolumntype{R}[1]{>{\PreserveBackslash\raggedleft}p{#1}}
\newcolumntype{L}[1]{>{\PreserveBackslash\raggedright}p{#1}}

\newcolumntype{I}{!{\vrule width 1pt}}
\newlength\savedwidth

\begin{document}


\baselineskip=15pt


\titlerunning{A nonlinear semigroup approach to  Hamilton-Jacobi equations}

\title{A nonlinear semigroup approach to Hamilton-Jacobi equations--revisited}

\author{Panrui Ni \and Lin Wang}

\date{\today}

\maketitle

\address{Panrui Ni: Shanghai Center for Mathematical Sciences,  Fudan University, Shanghai 200433, China; \email{prni18@fudan.edu.cn}
\and Lin Wang: School of Mathematics and Statistics, Beijing Institute of Technology, Beijing 100081, China; \email{lwang@bit.edu.cn}
}

\vspace{-2em}

\begin{abstract}
We consider the Hamilton-Jacobi equation
\[{H}(x,Du)+\lambda(x)u=c,\quad x\in M,
\]
where $M$ is a connected, closed and smooth Riemannian manifold.  The functions ${H}(x,p)$ and $\lambda(x)$ are continuous. ${H}(x,p)$  is convex, coercive with respect to  $p$, and $\lambda(x)$ changes the signs. The first breakthrough to this model was achieved by Jin-Yan-Zhao \cite{JYZ} under the Tonelli conditions. In this paper, we consider more detailed structure of the viscosity solution set and large time behavior of the viscosity solution on the Cauchy problem.

\end{abstract}

%
\tableofcontents


\section{Introduction and main results}\label{sec1}
\setcounter{equation}{0}
\setcounter{footnote}{0}


Inspired by Aubry-Mather theory and weak KAM theory for classical Hamiltonian systems,  an action minimizing method for contact Hamiltonian systems was developed in a series of papers \cite{SWY,WWY,WWY1,WWY2,WWY3}. Let $H:T^*M\times\R\to\R$ be a contact Hamiltonian. It turns out that the dependence of $H$ on the contact variable $u$ plays a crucial role in exploiting the dynamics generated by $H$.
By using previous dynamical approaches, some progress on viscosity solutions of Hamilton-Jacobi (HJ) equations have been achieved \cite{SWY,WWY1,WWY3}. In particular, the structure of the set of solutions  can be sketched if $H$ is
uniformly Lipschitz in $u$
based on the works mentioned before. Shortly after
\cite{WWY1} occurred, \cite{erg} generalized the results to ergodic problems by using  PDE approachs.
More recently, for a class of HJ equations with non-monotone dependence on $u$, the first breakthrough was achieved by Jin-Yan-Zhao \cite{JYZ} under the Tonelli conditions. In that  work, they provided a  description of the solution set of the stationary equation (formulated as (\ref{hu}) below) and revealed a bifurcation phenomenon with respect to  the value $c$ in the right hand side, which opened a way to exploit further properties of viscosity solutions beyond well-posedness for HJ equations with non-monotone dependence on $u$. The main results in this paper are  motivated by  \cite{JYZ}.

Let us consider the stationary equation:
\begin{equation}\tag{E$_0$}\label{hu}
{H}(x,Du)+\lambda(x)u=c,\quad x\in M.
\end{equation}
Throughout this paper, we assume
$M$ is a closed, connected and smooth Riemannian manifold. { $D$ denotes the spacial gradient with respect to $x\in M$.} Denote by $TM$ and $T^*M$ the tangent bundle and cotangent bundle of $M$ respectively. Let $H:T^*M\rightarrow\mathbb R$ satisfy
\begin{itemize}
\item [\textbf{(C):}] $H(x,p)$ is continuous;

\item [\textbf{(CON):}] $H(x,p)$ is convex in $p$, for any $x\in M$;

\item [\textbf{(CER):}] $H(x,p)$ is coercive in $p$, i.e. $\lim_{\|p\|_x\rightarrow +\infty}H(x,p)=+\infty$, where $\|\cdot\|_x$ denotes the norms induced by $g$ on both $TM$ and $T^*M$.
\end{itemize}

Correspondingly, one has the Lagrangian associated to $H$:
\begin{equation*}
  L(x,\dot x):=\sup_{p\in T^*_xM}\{\langle \dot x,p\rangle_x-H(x,p)\},
\end{equation*}
where $\langle\cdot,\cdot\rangle_x$ represents the canonical pairing between $T_xM$ and $T^*_xM$. The Lagrangian $L(x,\dot x)$ satisfies the following properties:
\begin{itemize}
\item [\textbf{(LSC):}] $L(x,\dot x)$ is lower semicontinous in $\dot x$, and continuous on the interior of its domain $\textrm{dom}(L):=\{(x,\dot x)\in TM:\ L(x,\dot x)<+\infty\}$;

\item [\textbf{(CON):}] $L(x,\dot x)$ is convex in $\dot x$, for any $x\in M$.
\end{itemize}
We also assume {$\lambda(x)$ is continuous} and satisfies
\begin{itemize}
\item [\textbf{($\pm$):}] there exist $x_1,x_2\in M$ such that $\lambda(x_1)>0$ and $\lambda(x_2)<0$.
\end{itemize}
{Throughout this paper, we define
\begin{equation}\label{v-norm}
\lambda_0:=\|\lambda(x)\|_\infty>0,
\end{equation}}
where $\|\cdot\|_\infty$ stands for the supremum norm of the  functions on their domains.
From an economical point of view, the assumption ($\pm$) may be corresponding to  a model with fluctuating  interest rates. To be more precise, the rates
depend on the space variable and admit values above and below zero at some places. From a
theoretical point of view, it is a toy model of the HJ equations with non-monotone dependence on $u$. Based on this model, we revealed some different phenomena  from the cases with monotone dependence on $u$ can be revealed.
\begin{remark}
The model (\ref{hu}) has been considered in \cite{Z1}. In that paper, the function $\lambda(x)$ is non-negative and positive on the projected Aubry set of $H(x,p)$. In this case,  the solution of (\ref{hu}) is unique. The asymptotic behavior of the solution of (\ref{hu}) is also studied in \cite{Z1} when $\lambda_0 \to 0$. When $\lambda_0 \to 0$ and the assumption ($\pm$) holds, the family of solutions of (\ref{hu}) may diverge, one can refer to \cite{N} for an example.
\end{remark}

In \cite{NWY}, the well-posedness of the Lax-Oleinik semigroup was verified for contact HJ equations under very mild conditions. By virtue of that, we generalize the results in \cite{JYZ} to the cases from the Tonelli conditions to the assumptions (C), (CON) and (CER) above. Henceforth, for simplicity of notation, we omit the word ``viscosity", if it is not necessary to be mentioned.
\begin{proposition}[Generalization of \cite{JYZ}]\label{mth}
 Let
\[c_0:=\inf_{u\in C^\infty(M)}\sup_{x\in M}\bigg\{H(x,Du)+\lambda(x)u\bigg\}.\]
Then $c_0$ is finite. Given $c\geq c_0$, the $\|\cdot\|_{W^{1,\infty}}$-norm of  all subsolutions of (\ref{hu}) is bounded. Moreover,
\begin{itemize}
\item [(1)] (\ref{hu}) has a solution if and only if $c\geq c_0$;
\item [(2)] if $c>c_0$, then (\ref{hu}) has at least two solutions.
\end{itemize}
\end{proposition}
The definition of $c_0$ is inspired by \cite{CIPP}. In light of that, $c_0$ is called the critical value. { Now we consider the most general case
\[H(x,u(x),Du(x))=c,\quad x\in M,\]
where the Hamiltonian $H(x,u,p)$ is continuous, superlinear in $p$ and uniformly Lipschitz in $u$. It was pointed out in \cite{erg} that there is a constant $c\in\mathbb R$ such that the above equation has viscosity solutions. Here we give some examples on the set $\mathfrak C$ of all such $c$'s, which reveals the essential difference between the monotone cases and the non-monotone cases:
\begin{itemize}
\item for classical Tonelli Hamiltonian $H(x,p)$, the set $\mathfrak C=\{c_0\}$. The number $c_0$ is called the effective Hamiltonian or the Ma\~n\'e critical value;
\item for the discounted Hamilton-Jacobi equation, i.e., the Hamiltonian is of the form $\lambda u+H(x,p)$ with $\lambda>0$, the set $\mathfrak C=\mathbb R$, see for example \cite{DFIZ};
\item for the model (\ref{hu}) considered here, the set $\mathfrak C=[c_0,+\infty)$. Here we note that the non-emptiness of $\mathfrak C$ is proved under (CER) instead of $H(x,p)$ is superlinear in $p$. {In view of the existence result in \cite{erg}, it means Proposition \ref{mth} is a non-trivial generalization of \cite{JYZ}};
\item for the Hamiltonian periodically depending on $u$, i.e., $H(x,u+1,p)\equiv H(x,u,p)$, the set $\mathfrak C$ is a bounded closed interval, see \cite{NWY2}.
\end{itemize}
}

Different from the Tonelli case considered in \cite{JYZ},  some new ingredients are needed for {\it a priori} estimates of subsolutions under the assumptions (C), (CON) and (CER). Those estimates  will be provided in Section \ref{bd}. The remaining  parts of the proof of Proposition \ref{mth} are similar to the one in \cite{JYZ}. We  postpone it to  Appendix \ref{A} for consistency.

Motivated by Proposition \ref{mth}, we are devoted to exploiting  more detailed information of this model. First of all, we obtain
\begin{theorem}\label{mthee2}
Let $c\geq c_0$.
There exist the maximal element $u_{\max}$ and the minimal element $u_{\min}$ in the set of solutions of (\ref{hu}).
\end{theorem}
\begin{remark}\label{minfr}
The viscosity solutions are equivalent to backward weak KAM solutions in our settings (see \cite[Proposition D.4]{NWY}). In terms of the correspondence between backward and forward weak KAM solutions (see Proposition \ref{eqvvv} below), it follows from Theorem \ref{mthee2} that there exist the maximal and minimal forward weak KAM solutions of (\ref{hu}). We denote $u^+_{\min}$ (resp. $u^+_{\max}$) the minimal (resp. maximal) froward weak KAM solution of (\ref{hu}). By Proposition \ref{S1}(3)(4), there hold
\[u^+_{\min}\leq u_{\min}=\lim_{t\to+\infty}T_t^-u^+_{\min},\quad \lim_{t\to+\infty}T_t^+u_{\max}=u^+_{\max}\leq u_{\max}.\]
\end{remark}
{
\begin{figure}[htbp]
\small \centering
\includegraphics[width=9cm]{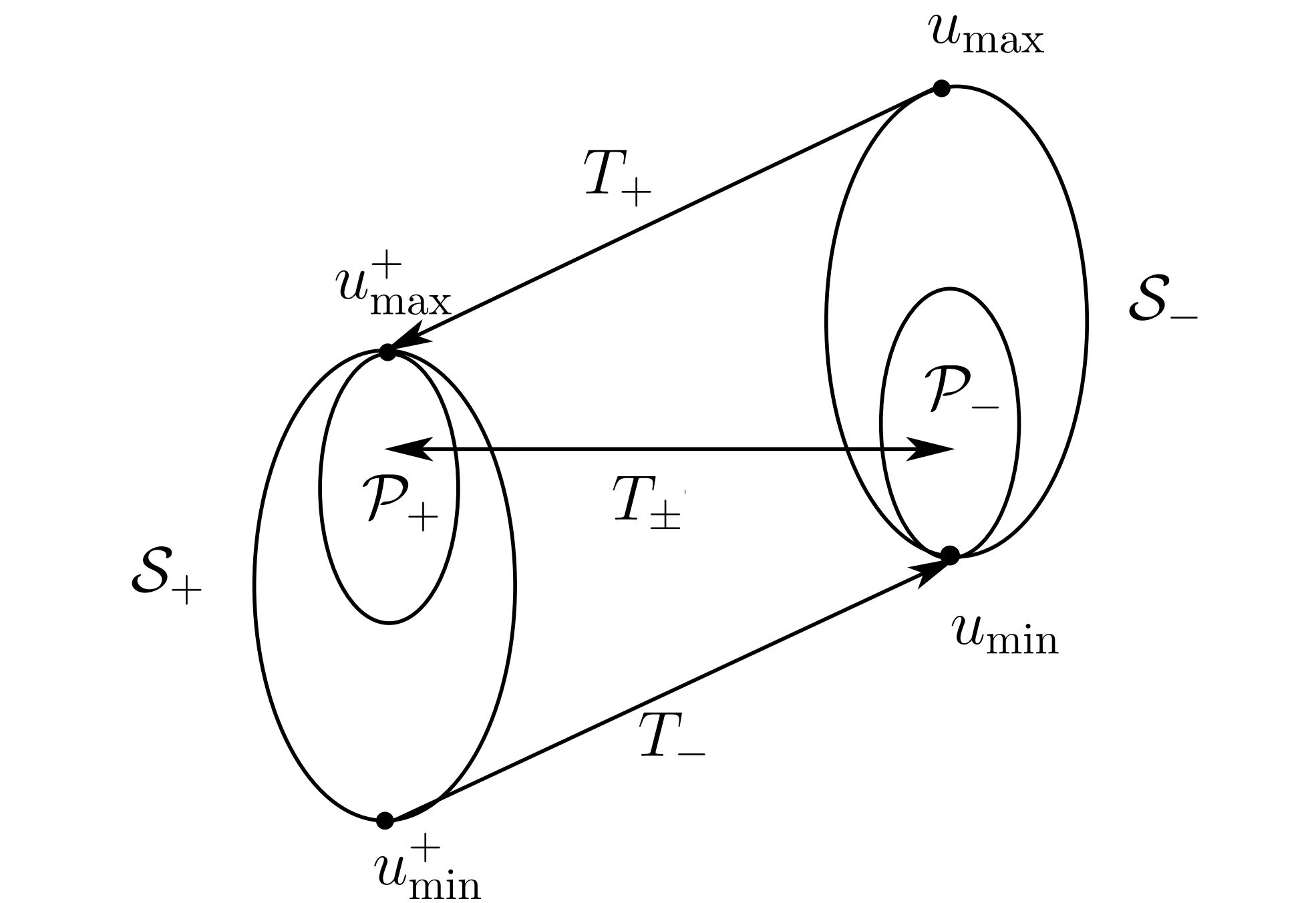}
\caption{The structure of the solution set of (\ref{hu})}\label{fig3}
\end{figure}
Let $\mathcal{S}_-$ (resp. $\mathcal{S}_+$) be the set of all backward (resp. forward) weak KAM solutions. Given $u_{\pm}\in \mathcal{S}_{\pm}$, if
\[u_-=\lim_{t\to\infty}T_t^-u_+,\quad u_+=\lim_{t\to\infty}T_t^+u_-,\]
then $u_-$ (resp. $u_+$) is called a conjugated backward (resp. forward) weak KAM solution.
See Figure \ref{fig3} for a rough description of structure of the solution set of (\ref{hu}) in general cases, where $\bold{T}_\pm:=\lim_{t\to\infty}T_t^{\pm}$, and $\mathcal{P}_-$ (resp. $\mathcal{P}_+$) denotes the set of all conjugated backward (resp. forward) weak KAM  solutions. }{ For further statement on conjugated weak KAM solutions, one can refer to \cite[Theorem 6.5 and Theorem 7.1]{Ish6}.}

By Proposition \ref{mth}(2), (\ref{hu}) has at least two solutions if $c>c_0$. Then a natural question is to figure out what happens if $c=c_0$.
In \cite{JYZ}, Jin, Yan and Zhao considered the following example:
\begin{example}\label{e1}
\begin{equation}\label{e1e}
  |u'(x)|^2+\sin x\cdot u(x)=c,\quad x\in \mathbb S^1\simeq [0,2\pi),
\end{equation}
where $\mathbb S^1$ denotes a flat circle with a fundamental domain $[0,2\pi)$.
\end{example}
It was shown that  $c_0=0$ and there are uncountably many solutions of (\ref{e1e}) in the critical case. A rough picture of certain solutions is given by Figure \ref{fig2}. See \cite[Theorem 3.5]{JYZ} for more details.
\begin{figure}[htbp]
\small \centering
\includegraphics[width=8cm]{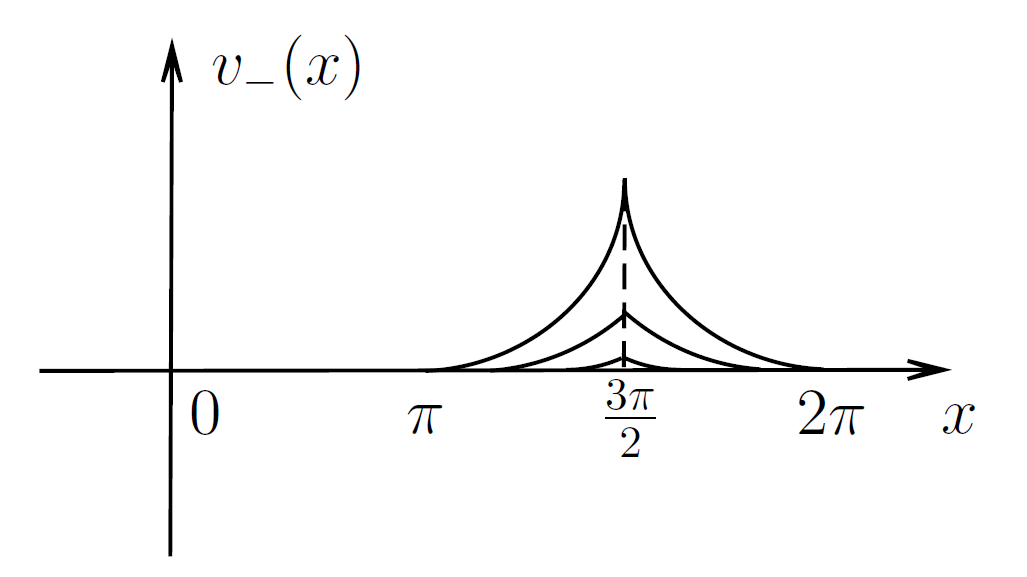}
\caption{Certain solutions of  (\ref{e1e}) with $c=0$}\label{fig2}
\end{figure}

{As a complement, we consider}
\begin{example}\label{e2}
\begin{equation}\label{e3e}
  \frac{1}{2}|u'(x)|^2+\sin x\cdot u(x)+\cos2x-1=c,\quad x\in \mathbb S^1\simeq [0,2\pi).
\end{equation}
\end{example}
We will prove that the critical value is also $c_0=0$, but  (\ref{e3e}) admits a unique solution in the critical case.  A rough picture of the solution is given by Figure \ref{fig}. See Remark \ref{geex} below for certain generalization of Example \ref{e2}.  Those two examples above show various possibilities about the solution set of  (\ref{hu})  in the critical case.

\begin{figure}[htbp]
\small \centering
\includegraphics[width=8cm]{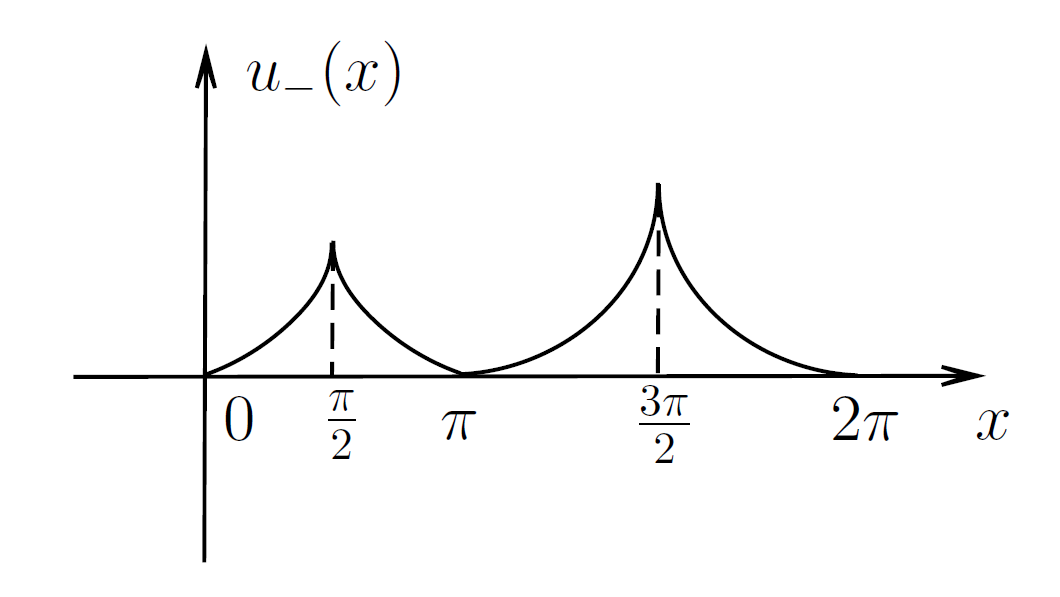}
\caption{The unique solution of  (\ref{e3e}) with $c=0$}\label{fig}
\end{figure}

In the second part, we consider  the evolutionary equation:
\begin{equation}\tag{CP}\label{C1}
  \left\{
   \begin{aligned}
   &\partial_t u(x,t)+H(x,Du(x,t))+\lambda(x)u(x,t)=c,\quad (x,t)\in M\times(0,+\infty).
   \\
   &u(x,0)=\varphi(x),\quad x\in M,
   \\
   \end{aligned}
   \right.
\end{equation}
where $\varphi\in C(M)$. It is well known that the viscosity solution of (\ref{C1}) is unique  (see \cite[Corollary 3.2]{Ish6} for instance). By \cite[Theorem 1]{NWY}, this solution can be represented by $u(x,t):=T_t^-\varphi(x)$, where $T^-_t: C(M)\rightarrow C(M)$ is defined implicitly by
\begin{equation}\tag{T-}\label{T-ee}
  T^-_t\varphi(x)=\inf_{\gamma(t)=x} \left\{\varphi(\gamma(0))+\int_0^t\bigl[L(\gamma(\tau),\dot{\gamma}(\tau))-\lambda(\gamma(\tau))T^-_\tau\varphi(\gamma(\tau))+c\bigl]{d}\tau\right\},
\end{equation}
where the infimum  is taken among absolutely continuous curves $\gm:[0,t]\rightarrow M$ with $\gamma(t)=x$.

In order to obtain equi-Lipschitz continuity of $\{T_t^-\varphi\}_{t\geq \delta}$ for a given $\delta>0$,  we have to strengthen the assumptions on $H$ from (CON), (CER) to the following:
\begin{itemize}
\item [\textbf{($\star$)}] $H(x,p)$ is strictly convex in $p$ for any $x\in M$, and there is a superlinear function $\theta:[0,+\infty)\to [0,+\infty)$ such that $H(x,p)\geq \theta(\|p\|)$.
\end{itemize}
Under the assumption ($\star$), the equi-Lipschitz continuity of $\{T_t^-\varphi\}_{t\geq \delta}$ follows from the locally Lipschitz property and boundedness of $T_t^-\varphi$ on $M\times (0,+\infty)$. From the weak KAM point of view, that kind of {locally Lipschitz property } can be verified by a standard procedure once we have the Lipschitz regularity of minimizers of $T_t^-\varphi(x)$ (see \cite[Lemma 4.6.3]{Fat-b}). However,  $H$ is only supposed to be continuous in our setting. Then one can not use the method of characteristics to  improve regularity of these minimizers. Following \cite{CCJWY}, we will deal with that issue by using the method of energy estimates. A key ingredient of that method is to establish the Erdmann condition for a non-smooth energy function. More precisely, we obtain the following result, whose proof is given in Appendix \ref{B}.
\begin{proposition}\label{mthlip}
Assume ($\star$) holds. If $T^-_t\varphi(x)$ has a bound independent of $t$, then  the family $\{T_t^-\varphi\}_{t\geq \delta}$ is equi-Lipschitz continuous, where $\delta$ is an arbitrarily positive constant.
\end{proposition}


Let us recall $u_{\max}$ denotes the maximal solution of (\ref{hu}), and $u^+_{\min}$ denotes its minimal froward weak KAM solution. By Remark \ref{minfr}, $u^+_{\min}\leq u_{\max}$ on $M$. Both of them play an  important role in characterizing the  large time behavior of the solution of (\ref{C1}).
By assuming ($\star$) holds, we obtain the following two results.
\begin{theorem}\label{mth2}
Let $u(x,t)$ be the solution of (\ref{C1}) with $c\geq c_0$. Then
\begin{itemize}
\item [(1)]  if the initial data $\varphi \geq u_{\max}$, then $u(x,t)$  converges to $u_{\max}$  uniformly on $M$ as $t\to +\infty$;
\item [(2)]  if there is a point $x_0\in M$ such that $\varphi(x_0)<u^+_{\min}(x_0)$, then $u(x,t)$ tends to $-\infty$ uniformly on $M$  as $t\to +\infty$.
\end{itemize}
\end{theorem}

\begin{theorem}\label{mth3}
Let $u(x,t)$ be the solution of (\ref{C1}) with $c>c_0$.  If the initial data $\varphi>u^+_{\min}$,  then $u(x,t)$ converges to $u_{\max}$ uniformly on $M$ as $t\to +\infty$.
\end{theorem}

\begin{remark}
{ For $\varphi\geq u^+_{\min}$}, if there exists $x_0\in M$ such that $\varphi(x_0)=u^+_{\min}(x_0)$, then $u(x,t)$ may not converge to $u_{\max}$.
\begin{itemize}
\item In  Example \ref{e1} with $c=0$, for each solution $v$ of (\ref{e1e}), it is easy to construct an initial data $\varphi$ satisfying { $\varphi\geq 0\geq u^+_{\min}$ and}
\[\{x\in M\ |\ \varphi(x)=u^+_{\min}(x)\}\neq \emptyset\]
such that $u(x,t)$ converges to $v$ uniformly on $M$. In fact, one can take $\varphi=v$ for instance.
\item For  Example \ref{e1} with $c=1$, by \cite[Theorem 3.14]{JYZ}, $u_{\min}=\sin x\neq u_{\max}$ and
\[\{x\in M\ |\ u_{\min}(x)=u^+_{\min}(x)\}\neq \emptyset.\]
Then one can take $\varphi=\sin x$ such that $u(x,t)$ converges to $u_{\min}$ uniformly on $M$.
\end{itemize}
\end{remark}


\vspace{1ex}

The rest of this paper is organized as follows. Section \ref{pre} gives some preliminaries on $T_t^{\pm}$, weak KAM solutions and Aubry sets. In Section \ref{bd}, {\it a priori} estimates on subsolutions of (\ref{hu}) are established.
The proof of  Theorem \ref{mthee2} and a detailed analysis of Example \ref{e2} are given in Section \ref{pthee2}.  Theorem \ref{mth2} and Theorem \ref{mth3} are proved in Section \ref{pth2} For the sake of completeness, some auxiliary results are proved in Appendix A.

\section{Preliminaries}\label{pre}

In this part, we collect some facts on $T_t^{\pm}$, weak KAM solutions and Aubry sets.  These facts hold under more general assumptions on the dependence of $u$.
Consider the evolutionary equation:
\begin{equation}\tag{A$_0$}\label{C}
  \left\{
   \begin{aligned}
   &\partial_t u(x,t)+H(x,u(x,t),Du(x,t))=0,\quad (x,t)\in M\times(0,+\infty).
   \\
   &u(x,0)=\varphi(x),\quad x\in M.
   \\
   \end{aligned}
   \right.
\end{equation}
and the stationary equation:
\begin{equation}\tag{B$_0$}\label{hj}
  H(x,u(x),Du(x))=0
\end{equation}

 { We denote by $(x,u,p)$ a point in $T^*M\times\mathbb R$, where $(x,p)\in T^*M$ and $u\in\mathbb R$.} Let $H:T^*M\times\mathbb R\rightarrow\mathbb R$ be a continuous Hamiltonian satisfying
\begin{itemize}
\item [\textbf{(CON):}] $H(x,u,p)$ is convex in $p$, for any $(x,u)\in M\times \mathbb R$;

\item [\textbf{(CER):}] $H(x,u,p)$ is coercive in $p$, i.e. $\lim_{\|p\|_x\rightarrow +\infty}(\inf_{x\in M}H(x,0,p))=+\infty$;

\item [\textbf{(LIP):}] $H(x,u,p)$ is Lipschitz in $u$, uniformly with respect to $(x,p)$, i.e., there exists $\Theta>0$ such that $|H(x,u,p)-H(x,v,p)|\leq \Theta|u-v|$, for all $(x,p)\in\ T^*M$ and all $u,v\in\mathbb R$.
\end{itemize}
Correspondingly, one has the Lagrangian associated to $H$:
\begin{equation*}
  L(x,u,\dot x):=\sup_{p\in T^*_xM}\{\langle \dot x,p\rangle-H(x,u,p)\}.
\end{equation*}
Due to the absence of superlinearity of $H$, the corresponding Lagrangian $L$ may take the value $+\infty$. Define
\[\text{dom}(L):=\{(x,\dot{x},u)\in TM\times\R\ |\ L(x,u,\dot{x})<+\infty\}.\]
Then $L(x,u,\dot x)$ satisfies the following properties:
\begin{itemize}
\item [\textbf{(LSC):}] $L(x,u,\dot x)$ is lower semicontinuous, and continuous on the interior of $\textrm{dom}(L)\times\mathbb R$;

\item [\textbf{(CON):}] $L(x,u,\dot x)$ is convex in $\dot x$, for any $(x,u)\in M\times \mathbb R$;

\item [\textbf{(LIP):}] $L(x,u,\dot x)$ is Lipschitz in $u$, uniformly with respect to $(x,\dot x)$, i.e., there exists $\Theta>0$ such that $|L(x,u,\dot x)-L(x,v,\dot x)|\leq \Theta|u-v|$, for all $(x,\dot x)\in\ \textrm{dom}(L)$ and all $u,v\in\mathbb R$.
\end{itemize}


\begin{proposition}\cite[Theorem 1]{NWY}\label{m1}
Both backward Lax-Oleinik semigroup
\begin{equation}\label{T-}
  T^-_t\varphi(x)=\inf_{\gamma(t)=x} \left\{\varphi(\gamma(0))+\int_0^tL(\gamma(\tau),T^-_\tau\varphi(\gamma(\tau)),\dot{\gamma}(\tau))\textrm{d}\tau\right\}
\end{equation}
and forward Lax-Oleinik semigroup
\begin{equation}\label{T+eq}
  T^+_t\varphi(x)=\sup_{\gamma(0)=x}\left\{\varphi(\gamma(t))-\int_0^tL(\gamma(\tau),T^+_{t-\tau}\varphi(\gamma(\tau)),\dot{\gamma}(\tau))d\tau\right\}.
\end{equation}
are well-defined for $\varphi\in C(M)$. The infimum  (resp. supremum) is taken among absolutely continuous curves $\gm:[0,t]\rightarrow M$ with $\gamma(t)=x$ (resp. $\gamma(0)=x$). If  $\varphi$ is  continuous, then $ u(x,t):=T^-_t\varphi(x)$ represents the unique  continuous viscosity solution of (\ref{C}). If $\varphi$ is Lipschitz continuous, then $u(x,t):=T^-_t\varphi(x)$ is also Lipschitz continuous on $M\times [0,+\infty)$.
\end{proposition}

\begin{proposition}\cite[Proposition 3.1]{NWY}\label{psg}
The Lax-Oleinik semigroups  have the following properties
\begin{itemize}
\item [(1)] For $\varphi_1$ and $\varphi_2 \in C(M)$, if $\varphi_1(x)< \varphi_2(x)$ for all $x\in M$, we have $T^-_t\varphi_1(x)< T^-_t\varphi_2(x)$ and $T^+_t\varphi_1(x)< T^+_t\varphi_2(x)$ for all $(x,t)\in M\times(0,+\infty)$.

\item [(2)] Given any $\varphi$ and $\psi\in C(M)$, we have $\|T^-_t\varphi-T^-_t\psi\|_\infty\leq e^{\Theta t}\|\varphi-\psi\|_\infty$ and $\|T^+_t\varphi-T^+_t\psi\|_\infty\leq e^{\Theta t}\|\varphi-\psi\|_\infty$ for all $t>0$.
\end{itemize}
\end{proposition}

Following Fathi \cite{Fat-b}, one can extend the definitions of backward and forward weak KAM solutions of  equation \eqref{hj} by using absolutely continuous calibrated curves instead of $C^1$ curves.
\begin{definition}\label{bws}
A function $u_-\in C(M)$ is called a backward weak KAM solution of \eqref{hj} if
\begin{itemize}
\item [(1)] For each absolutely continuous curve $\gamma:[t',t]\rightarrow M$, we have
\begin{equation*}
  u_-(\gamma(t))-u_-(\gamma(t'))\leq \int_{t'}^{t}L(\gamma(s),u_-(\gamma(s)),\dot \gamma(s))ds.
\end{equation*}
The above condition reads that $u_-$ is dominated by $L$ and denoted by $u_-\prec L$.

\item [(2)] For each $x\in M$, there exists an absolutely continuous curve $\gamma_-:(-\infty,0]\rightarrow M$ with $\gamma_-(0)=x$ such that
\begin{equation*}
  u_-(x)-u_-(\gamma_-(t))=\int_t^0L(\gamma_-(s),u_-(\gamma_-(s)),\dot \gamma_-(s))ds,\quad \forall t<0.
\end{equation*}
The curves satisfying the above equality are called $(u_-,L,0)$-calibrated curves.
\end{itemize}
\end{definition}
A forward weak KAM solution of \eqref{hj} can be defined in a similar manner.
Similar to \cite[Proposition 2.8]{WWY2}, one has
\begin{proposition}\label{rela}
	Let $\varphi\in C(M)$. Then
\begin{align}\label{sg}
-T^+_t(-\varphi)=\bar{T}^-_t\varphi,\quad -T^-_t(-\varphi)=\bar{T}^+_t\varphi,\quad \forall t\geq 0.
\end{align}
where $\bar{T}_t^{\pm}$ denote the Lax-Oleinik semigroups associated to {$L(x,-u,-\dot x)$.}
\end{proposition}

{The following two results are well known for Hamilton-Jacobi equations independent of $u$. They are also true in contact cases. We will prove them in Appendixes \ref{subsolution} and \ref{tt}. Proposition \ref{eudomii} provides some equivalent characterizations of Lipschitz subsolutions. Proposition \ref{T-T+>} shows that $T^+_t$ is a `weak inverse' of $T^-_t$.}

\begin{proposition}\label{eudomii}
Let $\varphi\in Lip(M)$. The following conditions are equivalent:
\begin{itemize}
\item [(1)] $\varphi$ is a Lipschitz subsolution of (\ref{hj});
\item [(2)] $\varphi\prec L$;
\item [(3)] for each $t\geq 0$,
\[T_t^-\varphi\geq \varphi\geq T_t^+\varphi.\]
\end{itemize}
\end{proposition}

{
\begin{proposition}\label{T-T+>}
For each $\varphi\in C(M)$, we have $T^+_t\circ T^-_t \varphi \leq \varphi\leq T^-_t\circ T^+_t \varphi$ for all $t\geq 0$.
\end{proposition}

}

The following three results come from \cite{NWY}, which give some connections among the fixed points of $T_t^{\pm}$, the lower (resp. upper) half limit, backward (resp. forward) weak KAM solutions and Aubry sets.
\begin{proposition}\cite[Proposition D.4]{NWY}\label{eqvvv}
Let $u_-\in C(M)$. The following statements are equivalent:
\begin{itemize}
\item [(1)]
$u_-$ is a fixed point of $T^-_t$;
\item [(2)] $u_-$ is a backward weak KAM solution of (\ref{hj});
\item [(3)] $u_-$ is a viscosity solution of (\ref{hj}).
\end{itemize}
Similarly, let $v_+\in C(M)$. The following statements are equivalent:
\begin{itemize}
\item [(1')] $v_+$ is a fixed point of $T^+_t$;
 \item [(2')] $v_+$ is a forward weak KAM solution of (\ref{hj});
 \item [(3)'] $-v_+$ is a viscosity solution of $H(x,-u(x),-Du(x))=0$.
\end{itemize}
\end{proposition}

\begin{proposition}\cite[Theorem 3 and Remark 3.5]{NWY}\label{S1}
Let $\varphi\in C(M)$.
\begin{itemize}
\item [(1)] If $T^-_t\varphi(x)$ has a bound independent of $t$, then the lower half limit
\[\check{\varphi}(x)=\lim_{r\rightarrow 0+}\inf\{T^-_t\varphi(y):\ d(x,y)<r,\ t>1/r\}\]
is a Lipschitz  solution of (\ref{hj}).
\item [(2)] If $T^+_t\varphi(x)$ has a bound independent of $t$, then the upper half limit
\[\hat{\varphi}(x)=\lim_{r\rightarrow 0+}\sup\{T^+_t\varphi(y):\ d(x,y)<r,\ t>1/r\},\]
is a Lipschitz  forward weak KAM solution of (\ref{hj}).
\item [(3)]  Let $u_-$ be a  solution of (\ref{hj}). Then  $T_t^+u_-\leq u_-$. The limit $u_+:=\lim_{t\rightarrow +\infty}T_t^+u_-$ exists, and $u_+$ is a forward weak KAM solution of (\ref{hj}).
\item [(4)]  Let $v_+$ be a forward weak KAM solution of (\ref{hj}). Then  $T_t^-v_+\geq v_+$. The limit $v_-:=\lim_{t\rightarrow +\infty}T_t^-v_+$ exists, and $v_-$ is a  solution of (\ref{hj}).
\end{itemize}
\end{proposition}

\begin{proposition}\cite[Theorem 3]{NWY}\label{abbb}
Let $u_-$ (resp. $u_+$) be a solution (resp. forward weak KAM solution) of (\ref{hj}).  We define the projected Aubry set with respect to $u_-$ by
\begin{equation*}
  \mathcal I_{u_-}:=\{x\in M:\ u_-(x)=\lim_{t\rightarrow +\infty}T^+_tu_-(x)\}.
\end{equation*}
Correspondingly, we define the projected Aubry set with respect to $u_+$ by
\begin{equation*}
  \mathcal I_{u_+}:=\{x\in M:\ u_+(x)=\lim_{t\rightarrow +\infty}T^-_tu_+(x)\}.
\end{equation*}
Both  $\mathcal I_{u_-}$ and $\mathcal I_{u_+}$ are nonempty. In particular, if $u_+(x)=\lim_{t\rightarrow +\infty}T^+_tu_-(x)$ and $u_-(x)=\lim_{t\rightarrow +\infty}T^-_tu_+(x)$, then
\[\mathcal I_{u_-}=\mathcal I_{u_+},\]
which is also denoted by $\mathcal I_{(u_-,u_+)}$, following the notation introduced by Fathi \cite{Fat-b}.
\end{proposition}

\section{Some estimates on subsolutions}\label{bd}

In this section, we assume the existence of subsolutions of (\ref{hu}) and prove some {\it a priori} estimates on subsolutions. The  existence of subsolutions  will be verified for $c\geq c_0$ in Proposition \ref{exist-sub} below.

Let $L(x,\dot{x})$ be the Legendre transformation of $H(x,p)$. Let $T_t^{\pm}$  be the Lax-Oleinik semigroups associated to
\[L(x,\dot{x})-\lambda(x)u(x)+c.\]
Similar to {\cite[Proposition 2.1]{Ish3}}, one can  prove the local boundedness of $L(x,\dot x)$ in a neighborhood of the zero section of $TM$.
\begin{lemma}\label{CL}
Let $H(x,p)$ satisfy (C)(CON)(CER), there exist constants $\delta>0$ and $C_L>0$ such that the Lagrangian $L(x,\dot x)$ associated to $H(x,p)$ satisfies
\begin{equation}\label{clxx}
L(x,\xi)\leq C_L,\quad \forall (x,\xi)\in M\times \bar B(0,\delta).
\end{equation}
\end{lemma}
Throughout this paper, we define
\begin{equation}\label{muuu}
\mu:=\textrm{diam}(M)/\delta,
\end{equation}
where $\textrm{diam}(M)$ denotes the diameter of $M$.
\begin{lemma}\label{ubd}
Let $\varphi\in C(M)$.
\begin{itemize}
\item [(1)]
$T^-_t\varphi$ has an upper bound independent of $t$;
\item [(2)] $T^+_t \varphi$ has a lower bound independent of $t$.
\end{itemize}
\end{lemma}

\begin{proof}
Take $x_1\in M$ with $\lambda (x_1)>0$. We first show
\begin{equation*}
T^-_t\varphi(x_1)\leq \max\bigg\{\varphi(x_1),\frac{L(x_1,0)+c}{\lambda(x_1)}\bigg\},\quad \forall t\geq 0.
\end{equation*}
Otherwise, there is $t>0$ such that
\begin{equation*}
T^-_t\varphi(x_1)> \max\bigg\{\varphi(x_1),\frac{L(x_1,0)+c}{\lambda(x_1)}\bigg\}\geq \frac{L(x_1,0)+c}{\lambda(x_1)}.
\end{equation*}
There are two cases:

(i) For all $s\in [0,t]$, we have
\begin{equation*}
T^-_s\varphi(x_1)> \frac{L(x_1,0)+c}{\lambda(x_1)}.
\end{equation*}
Take the constant curve $\gamma\equiv x_1$, we have
\begin{equation*}
T^-_t\varphi(x_1)\leq \varphi(x_1)+\int_0^t\bigg[L(x_1,0)+c-\lambda(x_1)T^-_s\varphi(x_1)\bigg]ds<\varphi(x_1),
\end{equation*}
which also leads to a contradiction.

(ii) There is $t_0\geq 0$ such that
\begin{equation*}
T^-_{t_0}\varphi(x_1)=\frac{L(x_1,0)+c}{\lambda(x_1)},
\end{equation*}
and
\begin{equation*}
T^-_s\varphi(x_1)> \frac{L(x_1,0)+c}{\lambda(x_1)},\quad \forall s\in (t_0,t].
\end{equation*}
Take the constant curve $\gamma\equiv x_1$, we have
\begin{equation*}
T^-_t\varphi(x_1)\leq T^-_{t_0}\varphi(x_1)+\int_0^t\bigg[L(x_1,0)+c-\lambda(x_1)T^-_s\varphi(x_1)\bigg]ds<\frac{L(x_1,0)+c}{\lambda(x_1)},
\end{equation*}
which leads to a contradiction.

We then prove that for all $x\in M$ and all $t>0$, $T_{t}^-\varphi(x)$ is bounded from above. It suffices to prove that for all $x\in M$ and  $t>0$, $T_{t+\mu}^-\varphi(x)$ is bounded from above, where $\mu$ is given by (\ref{muuu}). Let $\alpha:[0,\mu]\rightarrow M$ be a geodesic connecting $x_1$ and $x$ with constant speed, then $\|\dot \alpha\|\leq \delta$. Let
 \[K_0:=\max\bigg\{\varphi(x_1),\frac{L(x_1,0)+c}{\lambda(x_1)}\bigg\}.\]

Given $x\neq x_1$.  We assume $T^-_{t+\mu}\varphi(x)>K_0$. Otherwise the proof is completed. Since $T^-_{t}\varphi(x_1)\leq K_0$, there exists $\sigma\in[0,\mu)$ such that $T^-_{{t}+\sigma}\varphi(\alpha(\sigma))=K_0$ and  $T^-_{{t}+s}\varphi(\alpha(s))>K_0$ for all $s\in (\sigma,\mu]$. By definition
\begin{equation*}
\begin{aligned}
  T^-_{{t}+s}\varphi(\alpha(s))&\leq T^-_{{t}+\sigma}\varphi(\alpha(\sigma))+\int_\sigma^s \bigg[L(\alpha(\tau),\dot \alpha(\tau))-\lambda(\alpha(\tau))\cdot T^-_{{t}+\tau}\varphi(\alpha(\tau))+c\bigg]d\tau
  \\ &=K_0+\int_\sigma^s \bigg[L(\alpha(\tau),\dot \alpha(\tau))-\lambda(\alpha(\tau))\cdot T^-_{{t}+\tau}\varphi(\alpha(\tau))+c\bigg]d\tau,
\end{aligned}
\end{equation*}
which implies
\begin{equation*}
\begin{aligned}
  &T^-_{{t}+s}\varphi(\alpha(s))-K_0
  \leq \int_\sigma^s \bigg[L(\alpha(\tau),\dot \alpha(\tau))-\lambda(\alpha(\tau))\cdot T^-_{{t}+\tau}\varphi(\alpha(\tau))+c\bigg]d\tau
  \\ &\leq \int_\sigma^s \bigg[L(\alpha(\tau),\dot \alpha(\tau))-\lambda(\alpha(\tau))\cdot K_0+c\bigg]d\tau
  +\lambda_0\int_\sigma^s\bigg[T^-_{{t}+\tau}\varphi(\alpha(\tau))-K_0\bigg]d\tau
  \\ &\leq L_0\mu+\lambda_0\int_\sigma^s\bigg[T^-_{{t}+\tau}\varphi(\alpha(\tau))-K_0\bigg]d\tau,
\end{aligned}
\end{equation*}
where $\lambda_0$ is given by (\ref{v-norm}) and
\begin{equation*}
  L_0:=C_L+\lambda_0 K_0+c.
\end{equation*}
where $C_L$ is given by (\ref{clxx}).
By the Gronwall inequality, we have
\begin{equation*}
  T^-_{{t}+s}\varphi(\alpha(s))-K_0\leq L_0\mu e^{\lambda_0(s-\sigma)}\leq L_0\mu e^{\lambda_0\mu},\quad \forall s\in(\sigma,\mu].
\end{equation*}
Take $s=\mu$ we have $T^-_{{t}+\mu}\varphi(x)\leq K_0+L_0\mu e^{\lambda_0\mu}$.

Similar to the argument above, by choosing constant curve $\gamma(\tau)\equiv x_2$ with $\tau\in [0,t]$ and replacing $T_{t+\mu}^-\varphi$ by $T_{t-\mu}^+\varphi$, one has
\begin{equation}\label{tupp3}
T_{t}^+\varphi(x) \geq \min\bigg\{\varphi(x_2),\frac{L(x_2,0)+c}{\lambda(x_2)}\bigg\}-L_0\mu e^{\lambda_0\mu}.
\end{equation}
This completes the proof.
\end{proof}

\begin{corollary}\label{back-bd}
Let $u_0$ be a Lipschitz subsolution of (\ref{hu}). Then $T_{t}^{-}u_0$ (resp. $T_{t}^{+}u_0$)  has an upper (resp. lower) bound independent of $t$ and $u_0$.
\end{corollary}

\begin{proof}
We only prove $T_{t}^{-}u_0$  has an upper  bound independent of $t$ and $u_0$. The case with $T_{t}^{+}u_0$ is similar. Let
\begin{equation}\label{e_0}
\mathbf{e}_0:=\min_{(x,p)\in T^*M}H(x,p).
\end{equation}
By (CER), $\mathbf{e}_0$ is finite.
By the definition of the subsolution,
$H(x_1,p )+ \lambda(x_1)u_0(x_1) \leqslant c$ for any $p\in D^*u_0(x_1)$, where $D^*$ denotes the reachable gradients. It implies
 $$
 \lambda(x_1)u_0(x_1) \leqslant c-\min_{(x,p )\in T^*M}  H(x,p )= c-\mathbf{e}_0.
 $$
 Hence, for each subsolution $u_0$, we have
\begin{align*}
u_0(x_1)\leq \frac{c-\mathbf{e}_0}{\lambda(x_1)}.
\end{align*}
Let  \[K_0:=\frac{c-\mathbf{e}_0}{\lambda(x_1)},\quad L_0:=C_L+\lambda_0 K_0+c,\]
where $\lambda_0$ is given by (\ref{v-norm}). Here we note that
\[L(x_1,0)+c=\sup_{p\in T^*_xM}(-H(x_1,p))+c\leq -\min_{(x,p)\in T^*M}H(x,p)+c=c-\mathbf{e}_0.\]
By Lemma \ref{ubd}, we have
\begin{equation}\label{tupp2}
T_{t}^-u_0(x) \leqslant  K_0+L_0\mu e^{\lambda_0\mu}.
\end{equation}
This completes the proof.
\end{proof}

\begin{proposition}\label{uibo}
	There exists a constant $C>0$ such that for any subsolution $u$ of (\ref{hu}), there holds
      \[
      \|u\|_{W^{1,\infty}}\leqslant C.
      \]
\end{proposition}
\begin{proof}
 	By Proposition \ref{eudomii}, for each $t\geq 0$,
\[T_t^+u\leq u \leq T_t^-u.\]
 By Corollary \ref{back-bd}, there exist $C_1$, $C_2$ independent of $u$ such that
 \[C_2\leq u\leq C_1.\]

 For each $x,y\in M$, let $\alpha:[0,d(x,y)/\delta]\rightarrow M$ be a geodesic of length $d(x,y)$ with constant speed $\|\dot \alpha\|=\delta$ and connecting $x$ and $y$, where $d(x,y)$ denotes the distance between $x$ and $y$ induced by the Riemannian metric $g$ on $M$. Then
\begin{equation*}
  L(\alpha(s),\dot \alpha (s))\leq C_L,\quad \forall s\in [0,{d(x,y)/\delta}].
  \end{equation*}
By Proposition \ref{eudomii},
\begin{equation*}
\begin{aligned}
  u(y)-u(x)&\leq \int_0^{d(x,y)/\delta}\bigg[L(\alpha(s),\dot{\alpha}(s))-\lambda(\alpha(s))u(\alpha(s))+c\bigg]ds
  \\ &\leq \frac{1}{\delta}\bigg(C_L+\lambda_0\max\{|C_1|,|C_2|\}+c\bigg) d(x,y)=:\kappa d(x,y).
\end{aligned}
\end{equation*}
Note that $\kappa$ is independent of the choice of the subsolution $u$. We get the equi-Lipschitz continuity of $u$ by exchanging the role of $x$ and $y$.
\end{proof}

\begin{proposition}\label{exist-sol}
Let $u_0$ be a Lipschitz subsolution of (\ref{hu}). Then
\[
u_-:=\lim_{t\to +\infty}{ T_t^{-} }u_0(x), \quad u_+:=\lim_{t\to +\infty}{ T_{t}^{+}}u_0(x)
\]
exist, and the limit procedure is uniform in $x$. Moreover, $u_-$ is a  solution of (\ref{hu}), and $u_+$ is a forward weak KAM solution of (\ref{hu}). In particular, (\ref{hu}) has a solution $u_-$ for $c\geq c_0$.
\end{proposition}
\begin{proof}
We only prove that $u_-:=\lim_{t\to +\infty}{ T_t^{-} }u_0(x)$ exits, and it is a viscosity solution of (\ref{hu}). The existence of $u_+$ is similar . By Proposition \ref{S1}
\[\check{u}_-(x):=\lim_{r\rightarrow 0+}\inf\{T^-_tu_0(y):\ d(x,y)<r,\ t>1/r\}\]
is a solution of (\ref{hu}). By Proposition \ref{eudomii}(3) and Corollary \ref{back-bd}, for a given $x\in M$, the limit $\lim_{t\rightarrow+\infty}T^-_tu_0(x)$ exists. By definition, we have\[\check{u}_-(x)\leq \lim_{t\rightarrow+\infty}T^-_tu_0(x).\] Using Proposition \ref{eudomii}(3) again,  $T^-_tu_0$ is increasing in $t$ for all $t>0$, we have
\begin{equation*}
  \begin{aligned}
  T^-_t u_0(x)&=\lim_{r\rightarrow 0+} \inf\{T^-_tu_0(y):\ d(x,y)<r\}
  \\ &\leq \lim_{r\rightarrow 0+} \inf\{{T^-_{t+s}} u_0(y):\ d(x,y)<r,\ t+s>1/r\}=\check{u}_-(x).
  \end{aligned}
\end{equation*}
Then $\lim_{t\rightarrow+\infty}T^-_tu_0=\check{u}_-$. Note that $\check{u}_-$ is a  solution of (\ref{hu}).  By the Dini theorem, the family $\{T^-_t u_0\}_{t>0}$ uniformly converges to $\check{u}_-$.
\end{proof}

\section{Structure of the solution set of (\ref{hu})}\label{pthee2}
Let $\mathcal{S}_-$ (resp. $\mathcal{S}_+$) be the set of all solutions (resp. forward weak KAM solution) of (\ref{hu}).
\subsection{The maximal solution}\label{above2}

 We first prove the existence of the maximal solution. 
Since each solution is a subsolution, by Proposition \ref{uibo}, there are $C_1$ and $C_2$ such that $C_2\leq u_-\leq C_1$ for all $u_-\in\mathcal S_-$. Note that all solutions of (\ref{hu}) are fixed points of $T^-_t$. Take a {continuous} function $\varphi>C_1$ as the initial data. By Proposition \ref{psg} (1), $T^-_t\varphi$ is larger than every solution of (\ref{hu}). By Lemma  \ref{ubd}(1), $T^-_t\varphi$ has an upper bound independent of $t$. By Proposition \ref{S1} (1), the lower half limit
\[\check{\varphi}(x)=\lim_{r\rightarrow 0+}\inf\{T^-_t\varphi(y):\ d(x,y)<r,\ t>1/r\}\]
is a Lipschitz continuous viscosity solution of (\ref{hu}). Since $T^-_t\varphi$ is larger than every solution of (\ref{hu}), we have
\begin{equation*}
  \begin{aligned}
  \check{\varphi}(x)&=\lim_{r\rightarrow 0+}\inf\{T^-_t\varphi(y):\ d(x,y)<r,\ t>1/r\}
  \\ &\geq \lim_{r\rightarrow 0+}\inf\{v_-(y):\ d(x,y)<r\}=v_-(x),
  \end{aligned}
\end{equation*}
for all $v_-\in\mathcal S_-$. Thus, $\check{\varphi}(x)$ is the maximal solution of (\ref{hu}).

\subsection{The minimal solution}\label{below2}

  Since each forward weak KAM solution is dominated by $L(x,\dot x)-\lambda (x)u+c$, by Proposition \ref{eqvvv}, it is a subsolution of (\ref{hu}). By Proposition \ref{uibo}, there are $C_1$ and $C_2$ such that $C_2\leq u_+\leq C_1$ for all $u_+\in\mathcal S_+$. Take a {continuous} function $\varphi<C_2$ as the initial data. By Proposition \ref{psg} (1), $T^+_t\varphi$ is smaller than every forward weak KAM solution of (\ref{hu}). By Lemma  \ref{ubd}(2), $T^+_t\varphi$ has a lower bound independent of $t$. By Proposition \ref{S1} (2), the upper half limit
\[\hat{\varphi}(x)=\lim_{r\rightarrow 0+}\sup\{T^+_t\varphi(y):\ d(x,y)<r,\ t>1/r\}\]
is a forward weak KAM solution of (\ref{hu}). Since $T^+_t\varphi$ is smaller than every forward weak KAM solutions of (\ref{hu}), we have
\begin{equation*}
  \begin{aligned}
  \hat{\varphi}(x)&=\lim_{r\rightarrow 0+}\sup\{T^+_t\varphi(y):\ d(x,y)<r,\ t>1/r\}
  \\ &\leq \lim_{r\rightarrow 0+}\sup\{v_+(y):\ d(x,y)<r\}=v_+(x).
  \end{aligned}
\end{equation*}
for all $v_+\in\mathcal S_+$. Thus, $\hat{\varphi}(x)$ is the minimal forward weak KAM solution of (\ref{hu}). By Proposition \ref{S1} (4), $\hat{\varphi}_\infty:=\lim_{t\rightarrow +\infty}T^-_t \hat{\varphi}$ exists, and it is a solution of (\ref{hu}).  


\begin{lemma}\label{minlemm}
$\hat{\varphi}_\infty$ is the minimal solution of (\ref{hu}).
\end{lemma}
\begin{proof}
Define
\begin{equation*}
\mathcal P_-:=\{u_-\in\mathcal S_-:\ \exists u_+\in\mathcal S_+\ \textrm{such\ that}\ u_-=\lim_{t\rightarrow +\infty}T^-_tu_+\}.
\end{equation*}
We first prove that for each $v_-\in\mathcal P_-$, there holds $v_-\geq \hat{\varphi}_\infty$. In fact, by definition of $\mathcal P_-$, there is $u_+\in\mathcal S_+$ such that $v_-=\lim_{t\rightarrow +\infty}T^-_tu_+$. Since $\hat{\varphi}$ is the minimal forward weak KAM solution, we have \[u_+\geq \hat{\varphi}.\] Acting $T^-_t$ on both sides of the inequality above, and letting $t\rightarrow +\infty$, we have $v_-\geq \hat{\varphi}_\infty$.

We then prove that for each $v_-\in \mathcal S_-\backslash \mathcal P_-$,  $v_-\geq \hat{\varphi}_\infty$ still holds. Let $v_+:=\lim_{t\rightarrow +\infty} T^+_t v_-$ and $u_-:=\lim_{t\rightarrow +\infty} T^-_t v_+$. Then $u_-\in\mathcal P_-$, which implies $u_-\geq \hat{\varphi}_\infty$. By Proposition \ref{S1} (3), $v_+\leq v_-$. Then we have $T^-_t v_+\leq T^-_t v_-=v_-$. Taking  $t\rightarrow +\infty$ we get $u_-\leq v_-$. Therefore, $v_-\geq u_-\geq \hat{\varphi}_\infty$.
\end{proof}
So far, we complete the proof of Theorem \ref{mthee2}.
\subsection{On Example (\ref{e2})}\label{ex}

The Hamiltonian of (\ref{e3e}) is formulated as
\begin{equation}\label{Hex}
  H(x,u,p)=\frac{p^2}{2}+\sin x\cdot u+\cos 2x-1.
\end{equation}
We first show $c_0=0$. Assume (\ref{e3e}) admits a smooth subsolution $u_0$ when $c<0$, then we have $|u'_0(0)|^2\leq 2c<0$, which is impossible. When $c=0$, the constant function $\varphi\equiv 0$ is a subsolution of (\ref{e3e}). Therefore $c_0=0$. By Proposition \ref{exist-sol}, there is a solution $u_-$ of (\ref{e3e}) given by
\[u_-:=\lim_{t\to+\infty}T_t^-\varphi.\]
 Since $T_t^-\varphi\geq \varphi$, then $u_-\geq 0$.

We then divide the proof into the following steps:
\begin{itemize}
\item In Step 1, we discuss the dynamical behavior of the contact Hamiltonian flow $\Phi^H_t$ generated by $H(x,u,p)$, which is restricted on a two dimensional energy shell $M^0$.
      \subitem $\centerdot$ In Step 1.1, we show that the non-wandering set of $\Phi^H_t$ consists of four fixed points;
      \subitem $\centerdot$ In Step 1.2, we classify these fixed points by linearization;
      \subitem $\centerdot$ In Step 1.3, we show that for each solution $v_-$ of (\ref{e3e}), the $\alpha$-limit set of any $(v_-,L,0)$-calibrated curve $\gm:(-\infty,0]\rightarrow \mathbb S^1$ with $\gamma(0)\neq \pi/2$ and $3\pi/2$ can only be $0$ or $\pi$. We only focus on  the projected  $\alpha$-limit set defined on {$\mathbb S^1$}. For simplicity, we define
\[\alpha(\gamma):=\{{x\in \mathbb S^1}:\ \textrm{there\ exists\ a\ sequence}\ t_n\rightarrow-\infty\ \textrm{such\ that}\ {|\gamma(t_n)-x|\to 0}\},\]
      where  $\gamma:(-\infty,0]\rightarrow \mathbb S^1$ is a $(v_-,L,0)$-calibrated curve. Moreover, we check the constant curves $\gm(t)\equiv 0,\pi$ are calibrated curves, which implies $v_-(0)=v_-(\pi)=0$, $v'_-(0)=v'_-(\pi)=0$.

\item In Step 2,  we prove the uniqueness of the solution $v_-$ of (\ref{e3e}).
      \subitem $\centerdot$ In Step 2.1,  we prove that $v_-$ is unique near $0$ and $\pi$;
      \subitem $\centerdot$ In Step 2.2, we prove that $v_-$ is unique on $[\pi,2\pi)$ by the comparison along calibrated curves via the Gronwall inequality. The uniqueness of $v_-$ on $[0,\pi]$ is guaranteed by the comparison principle for the Dirichlet problem.
\end{itemize}

\noindent \textbf{Step 1. The dynamical behavior of the contact Hamiltonian flow.}

For each solution $v_-$ of (\ref{e3e}), let $\gm:(-\infty,0]\rightarrow \mathbb S^1$ be a $(v_-,L,0)$-calibrated curve. Similar to the analysis at the beginning of \cite[Section 3.2]{JYZ}, the derivative $v'_-(\gm(t))$ exists for each $t\in (-\infty, 0)$ and the orbit $(\gm(t),v_-(\gm(t)),v'_-(\gm(t)))$ satisfies the contact Hamilton equations generated by the Hamiltonian $H(x,u,p)$ defined in (\ref{Hex}). Then the proof of the uniqueness of the solution of (\ref{e3e}) is related to the contact Hamiltonian flow $\Phi^H_t$ generated by $H(x,u,p)$.

Since $c_0=0$ and $H(\gm(t),v_-(\gm(t)),v'_-(\gm(t)))=0$ for $t\in (-\infty, 0)$, we discuss the flow on the two dimensional energy shell
\[M^0:=\{(x,u,p)\in T^*\mathbb S^1\times\mathbb R:\ H(x,u,p)=0\}.\]
Note that along the contact Hamiltonian flow, we have $dH/dt=-H\partial H/\partial u$, which equals to zero on the set $M^0$. Thus, $M^0$ is an invariant set under the action of $\Phi^H_t$. Since we are interested in the orbit $(\gm(t),v_-(\gm(t)),v'_-(\gm(t)))$, we then consider the flow $\Phi^H_t$ restrict on $M^0$. The contact Hamilton equations then reduce to
\begin{equation}\label{He}
  \left\{
   \begin{aligned}
   &\dot x=p,
   \\
   &\dot p=-(\cos x\cdot u-2\sin 2x)-\sin x \cdot p,
   \\
   &\dot u=p^2.
   \\
   \end{aligned}
   \right.
\end{equation}

\noindent \textbf{Step 1.1. The non-wandering set.} We first consider the non-wandering set $\Omega$ of $\Phi^H_t|_{M^0}$. Suppose there is an orbit $(x(t),u(t),p(t))$ belongs to $\Omega$. Since $\dot u=p^2\geq 0$, $u(t)$ equals to a constant $c_u$ and $p(t)\equiv 0$. By $\dot x(t)=p(t)=0$, $x(t)$ also equals to a constant $c_x$. By $H(x,u,p)=0$ and $p=0$, we have
\[\sin x\cdot u+\cos 2x-1=0.\]
By $p=0$ and $\dot p=0$ we have
\[\cos x\cdot u-2\sin 2x=0.\]
A direct calculation shows that the only non-wandering points are
\[P_1=(0,0,0),\quad P_2=(\pi,0,0),\quad P_3=(\frac{\pi}{2},2,0),\quad P_4=(\frac{3\pi}{2},-2,0).\]

\noindent \textbf{Step 1.2. The classification of fixed points.} We then consider the dynamical behavior of $\Phi^H_t|_{M^0}$ near the fixed points. { After a translation, we put the fixed points to be the origin.} Near the points $P_1$ and $P_2$, the linearised equation of (\ref{He}) is
\[\dot x=p,\quad \dot p=4x,\quad \dot u=0.\]
Thus, $P_1$ and $P_2$ are hyperbolic fixed points for the dynamical system $\Phi^H_t|_{M^0}$. Near the points $P_3$ and $P_4$, the linearised equations of (\ref{He}) are
\[\dot x=p,\quad \dot p=-2x-p,\quad \dot u=0\]
and
\[\dot x=p,\quad \dot p=-2x+p,\quad \dot u=0\]
respectively. Thus, $P_3$ is a stable focus, and $P_4$ is an unstable focus.

\noindent \textbf{Step 1.3. The $\alpha$-limit set of calibrated curves.} The $\alpha$-limit set of a $(v_-,L,0)$-calibrated curve $\gm$ is contained in the projection of $\Omega$. If $\gm$ itself is not a fixed point, and the $\alpha$-limit of $\gm$ is a focus, then there are two constants $t_1<t_2<0$ with $\gm(t_1)=\gm(t_2)$ such that $v'_-(\gm(t_1))\neq v'_-(\gm(t_2))$, which is impossible. In other words, the obits near a focus can not form a 1-graph. Thus, the $\alpha$-limit of $\gm:(-\infty,0]\rightarrow \mathbb S^1$ with $\gamma(0)\neq \pi/2,3\pi/2$ can only be either $0$ or $\pi$. 
For constant curve $\gm:(-\infty,0]\rightarrow\mathbb S^1$ with $\gm(t)\equiv x_0$ and $x_0$ equals to either $0$ or $\pi$, we have
\[v_-(x_0)-v_-(x_0)=0=\int_0^tL(x_0,v_-(x_0),0)ds,\]
where
\[L(x,u,\dot x)=\frac{\dot x^2}{2}-\sin x\cdot u-\cos 2x+1\]
is the Lagrangian corresponding to $H(x,u,p)$. Then the constant curve $\gm$ is a $(v_-,L,0)$-calibrated curve. We then have
\[\lim_{t\rightarrow-\infty}v_-(\gm(t))=v_-(0)=v_-(\pi)=c_u=0,\]
and
\[\lim_{t\rightarrow-\infty}v'_-(\gm(t))=v'_-(0)=v'_-(\pi)=0.\]

\noindent \textbf{Step 2. The uniqueness of the solution $v_-$ of (\ref{e3e}).}

\noindent \textbf{Step 2.1.} 
For $x\in \mathbb S^1\backslash\{\pi/2, 3\pi/2\}$, let $\gm:(-\infty,0]\rightarrow\mathbb S^1$ with $\gm(0)=x$ be a $(v_-,L,0)$-calibrated curve.
We claim that there is a constant $\delta>0$ such that for $x\in[0,\delta]$, the $\alpha$-limit of the calibrated curve $\gm$ is $0$. If not, the $\alpha$-limit of $\gm$ is $\pi$ for all $x\in (0,\pi]$. Then $v_-$ is decreasing on $(0,\pi]$, since $v_-$ is increasing along $\gm$ by the last equality of (\ref{He}). By Step 1.3, $v_-(0)=v_-(\pi)=0$,
we get $v_-\equiv 0$ on $[0,\pi]$, which is impossible. By similar arguments, we conclude that there is a constant $\delta>0$ such that the $\alpha$-limit of $\gm$ is $0$ for $x\in[0,\delta]\cup[2\pi-\delta,2\pi)$, and the $\alpha$-limit of $\gm$ is $\pi$ for $x\in[\pi-\delta,\pi+\delta]$. Shrink $\delta$ if necessary, the 1-graph $(x,v_-(x),v'_-(x))$ coincides with the local unstable manifold of $P_1$ (resp. $P_2$) corresponding to the restricted flow $\Phi^H_t|_{M^0}$ when $x\in[0,\delta]\cup[2\pi-\delta,2\pi)$ (resp. $x\in[\pi-\delta,\pi+\delta]$). Therefore, the solution $v_-$ is unique on $[0,\delta]\cup[2\pi-\delta,2\pi)\cup[\pi-\delta,\pi+\delta]$.

\noindent \textbf{Step 2.2.} Since $\sin x\geq \sin\delta>0$ for $x\in [\delta,\pi-\delta]$, by the uniqueness of the solution of the Dirichlet problem (cf. \cite[Theorem 3.3]{CIL}), $v_-$ is unique on $[0,\pi]$. It remains to consider the uniqueness {of $v_-$ for $x\in [\pi,2\pi)$.}
Assume that there are two solutions $u_-$ and $v_-$ satisfying $u_-(x)>v_-(x)$ at some point $x\in (\pi+\delta,3\pi/2)$. Let $\gm$ be a $(v_-,L,0)$-calibrated curve with $\gm(0)=x$. Without any loss of generality, we assume the $\alpha$-limit of $\gm$ is $\pi$. {Take} $t_0<0$ such that $\gm(t_0)=\pi+\delta$, and {define}
\[G(s):=u_-(\gm(s))-v_-(\gm(s)),\quad s\in[t_0,0].\]
Then $G(t_0)=0$ and $G(0)>0$. By continuity, there is $\sigma_0\in[t_0,0)$ such that $G(\sigma_0)=0$ and $G(\sigma)>0$ for all $\sigma\in(\sigma_0,0]$. By definition we have
\begin{equation*}
  u_-(\gm(\sigma))-u_-(\gm(\sigma_0))\leq \int_{\sigma_0}^\sigma L(\gm(s),u_-(\gm(s)),\dot \gm(s))ds,
\end{equation*}
and
\begin{equation*}
  v_-(\gm(\sigma))-v_-(\gm(\sigma_0))=\int_{\sigma_0}^\sigma L(\gm(s),v_-(\gm(s)),\dot \gm(s))ds,
\end{equation*}
which implies
\begin{equation*}
  G(\sigma)\leq \int_{\sigma_0}^\sigma G(s)ds.
\end{equation*}
By the Gronwall inequality, we have $G(\sigma)\equiv 0$ for all $\sigma\in(\sigma_0,0]$, which contradicts $u_-(x)>v_-(x)$. The case $x\in (3\pi/2,2\pi-\delta)$ is similar. By the continuity of $v_-$ at $3\pi/2$, we finally conclude that the solution is unique on $[\pi,2\pi)$.

{
\begin{remark}\label{geex}
The method introduced in this section can be generalized to the following case
\[H(x,Du)+\lambda(x)u=c,\quad x\in\mathbb S^1,\]
where $\lambda(x)$ and $H(x,p)$ are of class $C^3$ and
\begin{itemize}
\item [(i)] the zero points of $\lambda(x)$ are $x_1$ and $x_2$, and $\lambda'(x)\neq 0$ at $x_1$ and $x_2$;
\item [(ii)] $H(x,p)$ is strictly convex and superlinear in $p$, $H(x,p)\equiv H(x,-p)$, \[\max_{x\in \mathbb S^1}H(x,0)=0\] and the maximum is achieved at $x_1$ and $x_2$, and the Hessian matrix of $H$
is negative definite at $(x_1,0)$ and $(x_2,0)\in T^*\mathbb S^1$;
\item [(iii)] for all $x\in\mathbb S^1$, let $\gamma:(-\infty,0]\to \mathbb S^1$ with $\gamma(0)=x$ be a calibrated curve, then the $\alpha$-limit of $\gamma$ is either $x_1$ or $x_2$.
\end{itemize}
By (ii), $H(x,p)\geq H(x,0)$, where the equality holds if and only if $p=0$. By the argument at the beginning of this section, it is direct to see the critical value $c_0=0$. Now let $c=0$. The contact Hamilton equations for $\Phi^H_t|_{M^0}$ are
\begin{equation}\label{He2}
  \left\{
   \begin{aligned}
   &\dot x=\frac{\partial H}{\partial p}(x,p),
   \\
   &\dot p=-\frac{\partial H}{\partial x}(x,p)-\lambda'(x)u-\lambda(x) p,
   \\
   &\dot u=\frac{\partial H}{\partial p}(x,p)p.
   \\
   \end{aligned}
   \right.
\end{equation}
By (ii), $\dot u\geq 0$ and the equality holds if and only if $p=0$. By the second equation in (\ref{He2}), there is only one non-wandering point of $\Phi^H_t|_{M^0}$ over $x_1$ (resp. $x_2$)
\[P_1=(x_1,0,0) \quad \textrm{(resp.}\ P_2=(x_2,0,0)\textrm{)}\]
 Note that
\[L(x,0)=\sup_{p\in T^*\mathbb S^1}-H(x,p)=-\inf_{p\in T^*\mathbb S^1}H(x,p)=-H(x,0).\]
Similar to Step 1.3 above, we have $v_-(x_1)=v_-(x_2)=0$ for each solution $v_-$. Near the points $P_1$ and $P_2$, the linearised equation is
\[\dot x=\frac{\partial^2 H}{\partial x\partial p}x+\frac{\partial^2 H}{\partial p^2}p,\quad \dot p=-\frac{\partial^2 H}{\partial x^2}x-\frac{\partial^2 H}{\partial x\partial p}p,\quad \dot u=0,\]
By (ii), $P_1$ and $P_2$ are hyperbolic fixed points. By (iii) and $\dot u\geq 0$, the solution is unique near $x_1$ and $x_2$. The remaining proof is similar to Step 2.2 above, we omit it for brevity.
\end{remark}
}

\section{Large time behavior of the solution of (\ref{C1})}\label{pth2}

Let us recall $u_{\max}$ (resp. $u_{\min}^+$) be the maximal solution (resp. minimal forward weak KAM solution) of (\ref{hu}). These two solutions play an important role in characterizing the large time behavior of the solution of (\ref{C1}).

\subsection{Above the maximal solution}\label{above}

Let $\varphi\geq u_{\max}$. Then $T_t^-\varphi\geq u_{\max}$. Combining with Lemma \ref{ubd}(1), $T^-_t\varphi(x)$ has a bound independent of $t$. Then the pointwise limit
\[\bar u(x):=\limsup_{t\rightarrow+\infty} T^-_t\varphi(x)\]
exists.

Assume ($\star$) holds. By Proposition \ref{mthlip}, the family $\{T^-_t\varphi(x)\}_{t\geq 1}$ is equi-Lipschitz in $x$. We denote by $\kappa$ the Lipschitz constant of $T^-_t\varphi(x)$ in $x$. Since
\begin{equation*}
  |\sup_{s\geq t} T^-_s\varphi(x)-\sup_{s\geq t} T^-_s\varphi(y)|\leq \sup_{s\geq t}|T^-_s\varphi(x)-T^-_s\varphi(y)|\leq \kappa d(x,y),
\end{equation*}
the limiting procedure
\begin{equation*}
  \bar{u}(x)=\lim_{t\rightarrow+\infty}\sup_{s\geq t} T^-_s\varphi(x)
\end{equation*}
is uniform in $x$. Thus, the function $\bar{u}(x)$ is Lipschitz continuous.
We assert that $\bar u$ is a subsolution. If the assertion is {true, by }Proposition \ref{exist-sol}, $\lim_{t\rightarrow +\infty} T^-_t\bar u(x)$ exits, and it is a solution. {Since} $T^-_t\varphi\geq u_{\max}$, we have $\bar u\geq u_{\max}$. Thus,  $\lim_{t\rightarrow+\infty} T^-_t\bar{u}=u_{\max}$. Based on Section \ref{above2},  the lower half limit $\check{\varphi}=u_{\max}$. By the definition of $\check{\varphi}$, we have \[\liminf_{t\rightarrow+\infty} T^-_t\varphi(x)\geq \check{\varphi}(x)=u_{\max}.\] On the other hand,
\[\limsup_{t\rightarrow+\infty}T^-_t\varphi(x)=\bar u(x)\leq \lim_{t\rightarrow +\infty} T^-_t\bar u(x)=u_{\max}(x).\]
It follows that $\lim_{t\rightarrow+\infty}T^-_t\varphi=u_{\max}$ uniformly on $M$.

It remains to prove $\bar u$ is a subsolution. By Proposition \ref{eudomii}, we only need to show $T^-_t\bar u$ is {increasing} in $t$.

We claim that for every $\varepsilon>0$, there exists a constant $s_0>0$ independent of $x$ such that for any $s\geq s_0$,
\[T^-_s\varphi(x)\leq \bar u(x)+\varepsilon.\]
Fixing $x\in M$, by definition of $\limsup$, for every $\varepsilon>0$, there is $s_0(x)>0$ such that for any $s\geq s_0(x)$,
\[T^-_s\varphi(x)\leq \bar u(x)+\frac{\varepsilon}{3}.\]
Take $r:=\frac{\varepsilon}{3\kappa}$. For $s\geq s_0(x)$, we have
\begin{equation*}
\begin{aligned}
  T^-_s\varphi(y)&\leq T^-_s\varphi(x)+\kappa d(x,y)\leq \bar u(x)+\frac{\varepsilon}{3}+\kappa d(x,y)
  \\ &\leq \bar u(y)+\frac{\varepsilon}{3}+2\kappa d(x,y)\leq \bar u(y)+\varepsilon,\quad \forall y\in B_r(x).
\end{aligned}
\end{equation*}
Since $M$ is compact, there are finite points $x_i\in M$ such that for each $y\in M$, there is a point $x_i$ such that $y\in B_r(x_i)$. Let $s_0:=\max_i s_0(x_i)$ and the claim is proved.

By Proposition \ref{psg}, for each $t>0$ we have
\[T^-_t(T^-_s\varphi(x))\leq T^-_t(\bar u(x)+\varepsilon)\leq T^-_t\bar u(x)+\varepsilon e^{\lambda_0 t},\]
where $\lambda_0:=\|\lambda(x)\|_\infty>0$. Taking the limit $s\rightarrow+\infty$, we have
\[\bar u(x)=\limsup_{s\rightarrow +\infty}T^-_t(T^-_s\varphi(x))\leq T^-_t\bar u(x)+\varepsilon e^{\lambda_0 t}.\]
Letting $\varepsilon\rightarrow 0+$, we get $\bar u(x)\leq T^-_t \bar u(x)$, which means $T^-_t\bar u(x)$ is increasing in $t$.


\subsection{Below the minimal solution}\label{below}

We have proved that for each $\varphi\geq u_{\max}$, $\lim_{t\rightarrow+\infty}T^-_t\varphi=u_{\max}$ uniformly on $M$. Combining with Proposition \ref{rela} and Proposition \ref{eqvvv}, one has
\begin{lemma}\label{vapp}
Let $\varphi\in C(M)$. If $\varphi\leq u^+_{\min}$, then $\lim_{t\rightarrow+\infty}T^+_t\varphi=u^+_{\min}$ uniformly on $M$.
\end{lemma}

\begin{lemma}\label{55}
Let $\varphi\in C(M)$ and there is a point $x_0\in M$ such that $\varphi(x_0)<u_{\min}^+(x_0)$, then $T^-_t\varphi(x)$ tends to $-\infty$ uniformly on $M$ as $t\rightarrow+\infty$.
\end{lemma}
\begin{proof}
We first prove that $\min_{x\in M}T^-_t\varphi(x)$ tends to $-\infty$ as $t\rightarrow+\infty$. We argue by contradiction. Assume there is a constant $K_1$ and a sequence $\{t_n\}_{n\in \mathbb{N}}$ such that $T^-_{t_n}\varphi\geq K_1$. By Lemma \ref{ubd}, $T^-_{t_n}\varphi$ also has a upper bound independent of $t$. Thus, the function $v_n(x):=T^-_{t_n}\varphi(x)$ is bounded continuous for each $n$. By Proposition \ref{T-T+>}, we have $\varphi(x_0)\geq T^+_{t_n}v_n(x_0)$. By Proposition \ref{uibo}, all of subsolutions are uniformly bounded. Denote by $K_2$ their lower bound.
Let $K':=\min\{K_1, K_2\}$, then $T^+_{t_n}v_n\geq T^+_{t_n}K'$. By Lemma \ref{ubd}(2), $T^+_t K'$ has a lower bound independent of $t$. Since $K'\leq K_2$, $T^+_t K'$ is smaller than every forward weak KAM solution of (\ref{hu}). By Lemma \ref{vapp},
 $\lim_{t\rightarrow +\infty}T^+_t K'$ exists and it equals to $u^+_{\min}$. We conclude
\begin{equation*}
  u^+_{\min}(x_0)\leq \limsup_{t_n\to+\infty} T^+_{t_n}v_n(x_0)\leq \varphi(x_0)<u^+_{\min}(x_0),
\end{equation*}
which leads to a contradiction.

We then prove that $T^-_t\varphi(x)$ tends to $-\infty$ uniformly as $t\rightarrow+\infty$. Let $W(x)$ be the inverse function of $x\mapsto xe^x$. Take $0<\eta\leq W(1)/\lambda_0$. We define $K(t):=\min_{x\in M}T^-_t\varphi(x)$, which tends to $-\infty$ as $t\to+\infty$. We take an arbitrary $x\in M$. If $T^-_{t+\eta}\varphi(x)\leq K(t)$, then the proof is finished. So we assume $T^-_{t+\eta}\varphi(x)>K(t)$. Let $x_t$ be the minimal point of $T^-_t\varphi$. Take a geodesic $\alpha:[0,\eta]\to M$ with $\alpha(0)=x_t$, $\alpha(\eta)=x$ and constant speed $\|\dot\alpha\|\leq \textrm{diam}(M)/\eta$. By continuity, there is $\sigma\in[0,\eta)$ such that $T^-_{t+\sigma}\varphi(\alpha(\sigma))=K(t)$ and $T^-_{t+s}\varphi(\alpha(s))>K(t)$ for all $s\in(\sigma,\eta]$. Then
\begin{equation*}
\begin{aligned}
  &T^-_{{t}+s}\varphi(\alpha(s))
  \leq T^-_{t+\sigma}\varphi(\alpha(\sigma))+\int_\sigma^s \bigg[L(\alpha(\tau),\dot \alpha(\tau))-\lambda(\alpha(\tau))\cdot T^-_{{t}+\tau}\varphi(\alpha(\tau))+c\bigg]d\tau
  \\ &\leq K(t)+\int_\sigma^s \bigg[L(\alpha(\tau),\dot \alpha(\tau))-\lambda_0 K(t)+c\bigg]d\tau
  +\lambda_0\int_\sigma^s\bigg[T^-_{{t}+\tau}\varphi(\alpha(\tau))-K(t)\bigg]d\tau
  \\ &\leq K(t)+\bar{C}_L\eta-\lambda_0\eta K(t)+\lambda_0\int_\sigma^s\bigg[T^-_{{t}+\tau}\varphi(\alpha(\tau))-K(t)\bigg]d\tau,
\end{aligned}
\end{equation*}
where
\[\bar{C}_L:=\max_{x\in M,\|\dot x\|\leq \textrm{diam}(M)/\eta}|L(x,\dot x)+c|\]
is finite for a fixed $\eta$ by the assumption ($\star$). By the Gronwall inequality, we have
\[T^-_{{t}+s}\varphi(\alpha(s))\leq \bar{C}_L\eta e^{\lambda_0\eta}+(1-\lambda_0\eta e^{\lambda_0\eta})K(t).\]
Since $\eta\leq W(1)/\lambda_0$, we have $1-\lambda_0\eta e^{\lambda_0\eta}>0$. Take $s=\eta$, we finally conclude that $T^-_t\varphi(x)$ tends to $-\infty$ as $t\to+\infty$.
\end{proof}
So far, we complete the proof of Theorem \ref{mth2}.

\subsection{Proof of Theorem \ref{mth3}}\label{pth3}


According to Proposition \ref{exist-sub}, for $c\geq c_0$, (\ref{hu}) has a Lipschitz subsolution. Let $u_0$ be a subsolution of (\ref{hu}) with $c=c_0$. For $c>c_0$, there holds   \[T_t^+u_0< u_0<T_t^-u_0.\] One can construct two different solutions $u_-$ and $v_-$ of (\ref{hu}) from $u_0$ by Proposition \ref{multi}. Precisely, we have
\begin{equation}\label{u0u-}
  u_-=\lim_{t\rightarrow +\infty} T^-_t u_0,\ u_+=\lim_{t\rightarrow +\infty} T^+_t u_0,\ \ v_-=\lim_{t\rightarrow +\infty} T^-_t u_+.
\end{equation}
It follows that  $u_+<u_0<u_-$.

\begin{lemma}\label{cb}
Let $c>c_0$. For each $\alpha\in (0,1]$ and each solution $w_-$ of (\ref{hu}), the convex combination
\[u_\alpha:=\alpha u_0+(1-\alpha)w_-\]
is a strict subsolution of (\ref{hu}). In particular, we have $T^+_tu_\alpha<u_\alpha<T^-_t u_\alpha$.
\end{lemma}
\begin{proof}
Since $u_0$ is a  Lipschitz subsolution of (\ref{hu}) with $c=c_0$, we have
\begin{equation*}
  H(x,Du_0(x))+\lambda (x)u_0(x)+(c-c_0)\leq c,\quad a.e\ x\in M.
\end{equation*}
Since $w_-$ is a solution of (\ref{hu}), we have
\begin{equation*}
  H(x,D w_-(x))+\lambda (x)w_-(x)=c,\quad a.e.\ x\in M.
\end{equation*}
Therefore
\begin{equation*}
\begin{aligned}
  \alpha H(x,Du_0(x))&+(1-\alpha)H(x,D w_-(x))
  \\ &+\lambda(x)\bigg(\alpha u_0(x)+(1-\alpha)w_-(x)\bigg)+\alpha(c-c_0)\leq c,\quad a.e.\ x\in M.
\end{aligned}
\end{equation*}
By the convexity of $H(x,p)$ with respect to $p$, the Jensen's inequality gives
\begin{equation*}
  H(x,Du_\alpha(x))+\lambda(x) u_\alpha(x)\leq (1-\alpha)c+\alpha c_0,\quad a.e.\ x\in M.
\end{equation*}
Let $\epsilon_0:=\alpha(c-c_0)>0$. Then
\begin{equation*}
  H(x,Du_\alpha(x))+\lambda(x) u_\alpha(x)+\epsilon_0\leq c,\quad a.e.\ x\in M.
\end{equation*}
By Lemma \ref{strss}, $T^+_tu_\alpha<u_\alpha<T^-_t u_\alpha$.
\end{proof}

\begin{lemma}\label{6.2}
Let $c>c_0$. Define $u_-$ and $v_-$ as in (\ref{u0u-}). Then $u_-$ is the maximal solution of (\ref{hu}), and $v_-$ is the minimal solution of (\ref{hu}).
\end{lemma}
\begin{proof}
We first prove that there is no solution $w_-$ different from $u_-$ such that $w_-\geq u_-$. Assume that there is such a solution $w_-$. Since $u_0<u_-\leq w_-$, there is $\alpha\in (0,1)$ such that $u_\alpha=\alpha u_0+(1-\alpha)w_-$ satisfies
\[\min_{x\in M}(u_-(x)-u_\alpha(x))=0.\]
Let $x_0\in M$ be the point at which the above minimum is attained.
 Then
\[T^-_t u_{\alpha}\leq T^-_t u_-.\]
 By Lemma \ref{cb}, we have $T^-_tu_\alpha(x_0)>u_\alpha(x_0)=u_-(x_0)=T^-_t u_-(x_0)$, which leads to a contradiction.

We then prove that $w_-\leq u_-$ for all solutions $w_-$. Assume that there is a solution $w_-$ such that
\[\max_{x\in M}(w_-(x)-u_-(x))>0.\]
Let $y_0\in M$ be the point at which the above maximum is attained.
Then the function $\bar u(x):=\max\{u_-(x),w_-(x)\}$ is a subsolution. By Proposition \ref{exist-sol}, we get a solution $\bar w_-:=\lim_{t\rightarrow+\infty}T^-_t\bar u$. We also have \[\bar w_-(y_0)\geq \bar u(y_0)=w_-(y_0)> u_-(y_0),\] which contradicts that $u_-$ is the maximal solution of (\ref{hu}).

Similar to the argument above, we conclude that $u_+$ is the minimal forward weak KAM solution of (\ref{hu}). By Lemma \ref{minlemm}, $v_-$ is the minimal solution of (\ref{hu}).
\end{proof}

Let us recall $u_0$ is a subsolution of (\ref{hu}) with $c=c_0$. For $c>c_0$, there holds   \[T_t^+u_0< u_0<T_t^-u_0.\]
By Proposition \ref{psg}(1) and Proposition \ref{T-T+>}, we have
\[T^-_{t+s}u_0\geq T^-_{t+s}\circ T^+_tu_0=T^-_s\circ (T^-_t\circ T^+_t u_0)\geq T^-_s u_0\]
for all $t,s\geq 0$. Letting $s\rightarrow +\infty$, we have
\begin{equation}\label{umaxconv}
  \lim_{s\rightarrow+\infty}T^-_{t+s}\circ T^+_tu_0=u_{\max},
\end{equation}
for each $t>0$. Let $\varphi\in C(M)$ satisfy $u^+_{\min}<\varphi\leq u_{\max}$. Since $u^+_{\min}=\lim_{t\rightarrow+\infty}T^+_tu_0$ by Lemma \ref{6.2}, there is $t_0>0$ such that $T^+_{t_0}u_0\leq \varphi$ on $M$. Then we have
\[T^-_{t_0+s}\circ T^+_{t_0}u_0\leq T^-_{t_0+s}\varphi\leq u_{\max}.\]
Letting $s\rightarrow +\infty$ and by (\ref{umaxconv}), we have \[\lim_{t\rightarrow +\infty}T^-_t\varphi=u_{\max}.\]
Now we assume ($\star$) holds. Then for each $\varphi>u^+_{\min}$, there is $\varphi_1$ and $\varphi_2$ such that
\[\varphi_1\geq u_{\max},\quad u^+_{\min}<\varphi_2\leq u_{\max},\quad \varphi_2\leq\varphi\leq \varphi_1.\]
Then we have $T^-_t\varphi_2\leq T^-_t\varphi\leq T^-_t\varphi_1$. Since $\lim_{t\to+\infty} T^-_t\varphi_i=u_{\max}$ for $i=1,2$, we have \[\lim_{t\rightarrow +\infty}T^-_t\varphi=u_{\max}.\]
The proof of Theorem \ref{mth3} is now complete.

\vskip 0.5cm

\noindent {\bf Acknowledgements:} The authors would like to thank Professor J. Yan for many helpful discussions.
Lin Wang is supported by NSFC Grant No. 12122109, 11790273.

\appendix
\section{Auxiliary results}
\subsection{Proof of Proposition \ref{eudomii}}\label{subsolution}

\begin{lemma}\label{SL}
If $\varphi$ is a Lipschitz subsolution of (\ref{hj}), then $\varphi\prec L$.
\end{lemma}
\begin{proof}
Without loss of generality, we assume $M$ is an open set of $\R^n$. In fact, for each absolutely continuous curve $\gm:[0,t]\rightarrow M$, we use a covering of it by local coordinate charts. Clearly, there exists $N\in \mathbb{N}$ such that $[0,t]=\cup_{i=0}^{N-1}[t_{i},t_{i+1}]$ with $t_0=0$, $t_{N}=t$, such that $\gm|_{[t_{i},t_{i+1}]}$ is contained in an open subset of $\mathbb R^n$.

By \cite[Proposition 2.4]{Ish3}, there is a function $q\in L^\infty([0,t],\mathbb R^n)$ such that for almost all $s\in[0,t]$, we have
\begin{equation*}
  \frac{d}{ds}\varphi(\gm(s))=q(s)\cdot \dot{\gamma}(s),
\end{equation*}
and the vector $q(s)$ belongs to $\partial_c \varphi(\gm(s))$. Here we recall the definition of the Clarke's generalized gradient
\[\partial_c\varphi(x):=\bigcap_{r>0}\overline{\textrm{co}}\{D\varphi(y):\ y\in B(x,r),\ \textrm{and}\ \varphi\ \textrm{is\ differentiable\ at}\ y\},\]
where $\overline{\textrm{co}}$ stands for the closure of the convex combination. Since $\varphi$ is a Lipschitz subsolution of (\ref{hj}), if $\varphi$ is differentiable at $y$, we have
\[H(y,\varphi(y),D\varphi(y))\leq 0.\]
By the convexity of $H$ with respect to $p$, and the definition of $\partial_c\varphi(x)$, we have
\[H(x,\varphi(x),q)\leq 0,\quad \forall q\in\partial_c\varphi(x).\]
We conclude that
\begin{align*}
\varphi(\gm(t))-\varphi(\gm(0))&=\int_0^t\frac{d}{ds}\varphi(\gm(s))ds=\int_0^t q(s)\cdot\dot{\gm}(s)ds
\\ &\leq \int_0^t\bigg[L(\gm(s),\varphi(\gm(s)),\dot{\gm}(s))+H(\gm(s),\varphi(\gm(s)),q(s))\bigg]ds
\\ &\leq \int_0^tL(\gm(s),\varphi(\gm(s)),\dot{\gm}(s))ds,
\end{align*}
which implies $\varphi\prec L$.
\end{proof}

\begin{lemma}\label{strss}
If $\varphi\prec L$, then for each $t\geq 0$, we have $T_t^-\varphi\geq \varphi\geq T_t^+\varphi$. Moreover, if there exists $\epsilon_0>0$ such that for a.e. $x\in M$,
\[H(x,u,Du)+\epsilon_0\leq 0.\]
then \[T_t^+\varphi< \varphi<T_t^-\varphi.\]
\end{lemma}
\begin{proof}

In the following, we only prove $T_t^-\varphi\geq \varphi$ for each $t\geq 0$, since the proof of  $T_t^+\varphi\leq \varphi$  is similar. By contradiction, we assume there exists $x_0\in M$ such that $\varphi(x_0)>T_t^-\varphi(x_0)$. Let $\gm:[0,t]\rightarrow M$ be a minimizer of $T_t^-\varphi$ with $\gm(t)=x_0$, i.e.
\begin{equation}
T_t^-\varphi(x)=\varphi(\gm(0))+\int_0^tL(\gm(\tau),T_\tau^- \varphi(\gm(\tau)),\dot{\gm}(\tau))d\tau.
\end{equation}
Let $F(\tau):=\varphi(\gm(\tau))-T_\tau^- \varphi(\gm(\tau))$. Since $F(t)>0$ and $F(0)=0$, then one can find $s_0\in [0,t)$ such that $F(s_0)=0$ and $F(s)>0$ for $s\in (s_0,t]$. A direct calculation shows
\[F(s)\leq \Theta\int_{s_0}^sF(\tau)d\tau,\]
which implies $F(s)\leq 0$ for $s\in (s_0,t]$ from the Gronwall inequality. It contradicts $F(t)>0$.

Next, we assume there exists $\epsilon_0>0$ such that for a.e. $x\in M$,
\[H(x,u,Du)+\epsilon_0\leq 0.\]
Let us denote
\[\tilde{L}(x,u,\dot{x}):=L(x,u,\dot{x})-\epsilon_0,\]
and let $\tilde{T}_t^-$ be the Lax-Oleinik semigroup associated to $\tilde{L}$. By a similar argument above, we have $\tilde{T}_t^-\varphi\geq \varphi$ and $\tilde{T}_t^+\varphi\leq \varphi$. Note that $\tilde{L}<L$. Using a similar argument as \cite[Proposition 3.1]{WWY1}, $\tilde{T}_t^-\varphi<{T}_t^-\varphi$ and $\tilde{T}_t^+\varphi>{T}_t^+\varphi$ for each $t>0$. Therefore, $T_t^-\varphi> \varphi$ and $T_t^+\varphi< \varphi$  for each $t> 0$. This completes the proof.
\end{proof}

\begin{lemma}
If for each $t>0$, $T^-_t\varphi\geq \varphi$, then $\varphi$ is a Lipschitz subsolution of (\ref{hu}).
\end{lemma}
\begin{proof}
Fix $T>0$, by assumption we have $T^-_t\varphi\geq \varphi$ for each $t\in[0,T]$. By \cite{NWY}, there is a constant $R_0>0$ depending on $T$ and $\|D\varphi\|_\infty$, such that $\|D T^-_t\varphi(x)\|_\infty\leq R_0$. Let {$R:=\max\{R_0,\|D\varphi\|_\infty\}$}, we make a modification
\[H_R(x,u,p):=H(x,u,p)+\max\{\|p\|^2-R^2,0\}.\]
Then $T^-_t\varphi$ is also the solution of (\ref{C}) with $H$ replaced by $H_R$. One can prove that the Lagrangian $L_R$ corresponding to $H_R$ is continuous. By the uniqueness of the solution of (\ref{C}), we have $T^-_t\varphi=T^R_t\varphi$, where $T^R_t\varphi$ is defined by (\ref{T-}) with $L$ replaced by $L_R$.

Let $\varphi$ be differentiable at $x\in M$. For each $v\in T_x M$, there is a $C^1$ curve $\gm:[0,T]\rightarrow M$ with $\gm(0)=x$ and $\dot \gm(0)=v$. By assumption for each $t\in [0,T]$, we have
\[\varphi(\gm(t))\leq T^-_t\varphi(\gm(t))=T^R_t\varphi(\gm(t))\leq \varphi(x)+\int_0^tL_R(\gm(s),T^R_s\varphi(\gm(s)),\dot \gm(s))ds.\]
Dividing by $t$ and let $t$ tend to zero, using the continuity of $\gm$, $L_R$ and $T^R_t\varphi(x)$. We get
\[D\varphi(x)\cdot v\leq L_R(x,\varphi(x),v).\]
Since $v$ is arbitrary, we have
\begin{equation*}
  H_R(x,\varphi(x),D\varphi(x))=\sup_{v\in T_x M}\bigg[D\varphi(x)\cdot v-L_R(x,\varphi(x),v)\bigg]\leq 0.
\end{equation*}
Therefore, $\varphi$ is a Lipschitz subsolution of \[H_R(x,u(x),Du(x))=0.\]
{By the definition of $H_R$}, $\varphi$ is also a Lipschitz subsolution of (\ref{hj}).
\end{proof}
\subsection{Proof of Proposition \ref{T-T+>}}\label{tt}

We only prove $\varphi\leq T^-_t\circ T^+_t \varphi$, the other side is similar. We argue by a contradiction. Assume that there is $x\in M$ and $t>0$ such that
\[T^-_t\circ T^+_t \varphi(x)< \varphi(x).\]
Let $\gm:[0,t]\rightarrow M$ with $\gm(t)=x$ be a minimizer of $T^-_t\circ T^+_t \varphi(x)$, and define
\[F(s):=T^+_{t-s}\varphi(\gm(s))-T^-_s\circ T^+_t \varphi(\gm(s)).\]
Then $F(0)=0$ and $F(t)>0$. By continuity, there is $\sigma\in[0,t)$ such that $F(\sigma)=0$ and $F(\tau)>0$ for all $\tau\in (\sigma,t]$. By definition, for $s\in(\sigma,t]$ we have
\begin{equation*}
\begin{aligned}
  T^-_s\circ T^+_t\varphi(\gm(s))&=T^-_\sigma\circ T^+_t\varphi(\gm(\sigma))+\int_\sigma^sL(\gm(\tau),T^-_\tau\circ T^+_t\varphi(\gm(\tau)),\dot\gm(\tau))d\tau
  \\ &=T^+_{t-\sigma}\varphi(\gm(\sigma))+\int_\sigma^s L(\gm(\tau),T^-_\tau\circ T^+_t\varphi(\gm(\tau)),\dot\gm(\tau))d\tau
  \\ &\geq T^+_{t-s}\varphi(\gm(s))-\int_\sigma^s L(\gm(\tau),T^+_{t-\tau}\varphi(\gm(\tau)),\dot\gm(\tau))d\tau
  \\ &\quad \quad\quad  +\int_\sigma^sL(\gm(\tau),T^-_\tau\circ T^+_t\varphi(\gm(\tau)),\dot\gm(\tau))d\tau
  \\ &\geq T^+_{t-s}\varphi(\gm(s))-\Theta\int_\sigma^s F(\tau)d\tau,
\end{aligned}
\end{equation*}
which implies
\[F(s)\leq \Theta\int_\sigma^s F(\tau)d\tau.\]
By the Gronwall inequality, we have $F(s)\equiv 0$ for $s\in[\sigma,t]$, which contradicts $F(t)>0$.

\subsection{Proof of Proposition \ref{mth}}\label{A}


\subsubsection{$c_0$ and subsolutions}\label{c0}
Inspired by \cite{CIPP}, we denote
\[c_0:=\inf_{u\in C^\infty(M)}\sup_{x\in M}\bigg\{H(x,Du)+\lambda(x)u\bigg\}.\]

\begin{proposition}\label{finite}
$c_0$ is finite.
\end{proposition}

\begin{proof}
Choose $u(x)\equiv 0$, then by definition,
\[
c_0\leq \sup_{x\in M}H(x,0)<+\infty.
\]
Let us recall
\begin{equation*}
\mathbf{e}_0:=\min_{(x,p)\in T^*M}H(x,p)>-\infty.
\end{equation*}
By the assumption ($\pm$), there exists $x_0\in M$ such that $\lambda(x_0)=0$. Thus for each $u\in C^\infty(M)$,
\begin{align*}
c_0=&\inf_{u\in C^\infty(M)}\sup_{x\in M}\,\,\bigg\{H(x,Du(x))+\lambda(x)u(x)\bigg\}\\
\geqslant&\,\inf_{u\in C^\infty(M)}\,\,\bigg\{H(x_0,Du(x_0))+\lambda(x_0)u(x_0)\bigg\}\\
=&\,\inf_{u\in C^\infty(M)}H(x_0,Du(x_0))\geqslant\mathbf{e}_0.
\end{align*}
This means $c_0$ is finite.
\end{proof}

\begin{proposition}\label{ns}
For $c<c_0$, (\ref{hu}) has no  continuous subsolutions.
\end{proposition}

\begin{proof} By contradiction, we assume for $c<c_0$, (\ref{hu}) admits a continuous subsolution $u:M\rightarrow\R$. By the definition of the subsolution,  for any $p\in D^+ u(x)$,
\begin{align*}
H(x,p)\leq c-\lambda(x)u(x)\leq c+\lambda_0\|u\|_\infty.
\end{align*}
Combining (CER), one can conclude that $u$ is Lipschitz continuous (see \cite[Proposition 1.14]{Ishii_chapter} for more details). By \cite[Lemma 2.2]{DFIZ},   for all $\varepsilon>0$, there exists $u_\varepsilon\in C^{\infty}(M)$ such that $\|u-u_\varepsilon\|_\infty<\varepsilon$ and for all $x\in M$,
\[
H(x,Du_\varepsilon(x))+\lambda(x)u(x)\leq c+\varepsilon.
\]
We choose $\varepsilon=\frac{1}{2(1+\lambda_0)}(c_0-c)>0$, then
\begin{align*}
&H(x,Du_\varepsilon(x))+\lambda(x)u_\varepsilon(x)\\
\leqslant&\,H(x,Du_\varepsilon(x))+\lambda(x)u(x)+\lambda_0\|u-u_\varepsilon\|_\infty\\
\leqslant&\,c+(1+\lambda_0)\varepsilon<c_0,
\end{align*}
this contradicts the definition of $c_0$.
\end{proof}

\subsubsection{Existence of subsolutions and solutions}\label{exis}

Let us recall that $T_t^{\pm}$  denote the Lax-Oleinik semigroups associated to
\[L(x,\dot{x})-\lambda(x)u(x)+c.\]

\begin{proposition}\label{exist-sub}
For $c\geq c_0$, (\ref{hu}) has a Lipschitz subsolution. Let $u_0$ be a subsolution of (\ref{hu}) with $c=c_0$. For $c>c_0$, there holds   \[T_t^+u_0< u_0<T_t^-u_0.\]
\end{proposition}

\begin{proof}
By the definition of $c_0$,  there exists $u_n\in C^\infty(M)$ such that for all $x\in M$,
\begin{equation}\label{aux-eq1}
H(x,Du_n(x))+ \lambda(x)u_n(x)\leqslant c_0+\frac{1}{n}.
\end{equation}
Namely, $u_{n}$ is a subsolution of
\[
H(x,Du)+\lambda(x)u={c_0+1},
\]

By Proposition \ref{uibo},
  $\{u_n \}_{n\geq1}$ is equi-bounded and equi-Lipschitz continuous.
Then by the Ascoli-Arzel\`a theorem, it contains a subsequence $\{u_{n_k} \}_{k\in \mathbb{N}}$ uniformly converging on $M$ to some $u_0\in$ Lip$(M)$. By the stability of subsolutions (see \cite[Theorem 5.2.5]{CS}), $u_0$ is a  subsolution of
\[
H(x,Du)+\lambda(x)u=c_0.
\]
Moreover, for $c>c_0$ and a.e. $x\in M$,  we have
\[
H(x,Du_0)+\lambda(x)u_0+(c-c_0)\leq c.
\]
By Lemma \ref{strss}, \[T_t^+u_0< u_0<T_t^-u_0.\]
This completes the proof.
\end{proof}


Combining Propositions \ref{ns}, \ref{exist-sub} and \ref{exist-sol}, we conclude that (\ref{hu}) has a solution if and only if $c\geq c_0$. It remains to prove the following result.
\begin{proposition}\label{multi}
(\ref{hu}) has at least two solutions for $c>c_0$.
\end{proposition}

\begin{proof}
By Proposition \ref{exist-sub}, if $c>c_0$, there exists a strict Lipschitz subsolution $u_0$ of (\ref{hu}). Based on Proposition \ref{eudomii}, for $t>0$,
\begin{equation}\label{eq:3}
T_t^{-}u_0(x)>u_0(x), \quad T_t^{+}u_0(x)<u_0(x).
\end{equation}
Denote
\begin{equation}\label{vis-def}
u_-:=\lim_{t\to +\infty}{ T_{t}^{-} } u_0(x), \quad u_+:=\lim_{t\to +\infty}{T_{t}^{+}}  u_0(x).
\end{equation}
and
\begin{equation}\label{ano-def}
v_-:=\lim_{t\to +\infty}T_{t}^{-}u_+(x).
\end{equation}
By Proposition \ref{exist-sol}, $u_-$ and $v_-$ are solutions of (\ref{hu}).

It remains to verify $u_-\neq v_-$. By contradiction, we assume $u_-\equiv v_-$ on $M$. In view of (\ref{ano-def}), we have
\begin{equation}\label{ulim}
u_-=\lim_{t\to +\infty}T_{t}^{-}u_+(x).
\end{equation}
Based on (\ref{ulim}), it follows from Proposition \ref{abbb} that
\begin{equation}\label{ass}
\mathcal I_{u_+}:=\{x\in M:\ u_-(x)=u_+(x)\}\neq\emptyset.
\end{equation}
On the other hand,
from \eqref{eq:3} and \eqref{vis-def}, it follows that for any $x\in M$,
\begin{equation}\label{eq:4}
u_+(x)<u_0(x)<u_-(x),
\end{equation}
which implies
\[\mathcal I_{u_+}=\emptyset.\]
This contradicts (\ref{ass}).
\end{proof}

\subsection{Proof of Proposition \ref{mthlip}}\label{B}

Assume that $H(x,p)$ is continuous and satisfies the condition ($\star$).
Then the associated Lagrangian $L(x,\dot x)$ satisfies
\begin{itemize}
\item [\textbf{(CL):}] $L(x,\dot x)$ and $\frac{\partial L}{\partial\dot x}(x,\dot x)$ are continuous;

\item [\textbf{(CON):}] $L(x,\dot x)$ is convex in $\dot x$, for any $x\in M$;

\item [\textbf{(SL):}] there is a superlinear function $\eta(r)$ such that $L(x,\dot x)\geq \eta(\|\dot x\|)$.
\end{itemize}
With a slight modification, \cite[Theorem 2.2]{CCJWY} implies
\begin{lemma}\label{C.2}(Erdmann condition).
For each $(x,t)\in M\times(0,+\infty)$, let $\gamma:[0,t]\rightarrow M$ be a minimizer of $T^-_t\varphi(x)$. Set $u_1(s):=T^-_s\varphi(\gamma(s))$ with $s\in[0,t]$, and
\[E_0(s):=\frac{\partial L}{\partial \dot x}(\gamma(s),\dot\gamma(s))\cdot\dot\gamma(s)-L(\gamma(s),\dot\gamma(s)),\]
then
\[E(s):=e^{\int_0^s \lambda(\gamma(r)) dr} [E_0(s)+\lambda(\gamma(s))u_1(s)]\]
satisfies $\dot E(s)=0$ a.e on $[0,t]$.
\end{lemma}
Based on Lemma \ref{C.2}, we have
\begin{thm}\label{A.1}
The function $(x,t)\mapsto T_t^-\varphi(x)$ is locally Lipschitz on $M\times (0,+\infty)$. More precisely, given two positive constants $\delta$ and $T$ with $\delta<T$. For each $\varphi\in C(M)$ and $t\in[\delta,T]$, the Lipschitz constant of $T^-_t\varphi(x)$ depends only on $\|\varphi\|_\infty$, $\delta$ and $T$.
\end{thm}
\begin{proof}
\textbf{Step 1. Lipschitz estimate of minimizers.} Given $(x,t)\in M\times[\delta,T]$. In the following, we denote by $\gamma:[0,t]\rightarrow M$ a minimizer of $T^-_t\varphi(x)$. We focus on the Lipschitz regularity of the curve $\gamma$.
Note that $T^-_t(-\|\varphi\|_\infty)\leq T^-_t\varphi\leq T^-_t\|\varphi\|_\infty$, $T^-_t\varphi$ is bounded by a constant $K$ depends only on $\|\varphi\|_\infty$ and $T$. We then have
\begin{equation*}
\begin{aligned}
  K &\geq T^-_t\varphi(x)=\varphi(\gm(0))+\int_0^t\bigg[L(\gm(s),\dot\gm(s))-\lambda(\gm(s))T^-_s\varphi(\gm(s))\bigg]ds
  \\ &\geq -\|\varphi\|_\infty-\lambda_0 KT+\int_0^tL(\gm(s),\dot\gm(s))ds.
\end{aligned}
\end{equation*}
By (SL), there is a constant $D$ such that $L(\gamma(s),\dot\gamma(s))\geq \|\dot\gamma(s)\|+D$, then we have
\[K+(\lambda_0 K+|D|)T+\|\varphi\|_\infty\geq \int_0^t \|\dot\gamma(s)\|ds.\]
Thus, there is $s_0\in [0,t]$ such that $\|\dot\gamma(s_0)\|$ is bounded by a constant depends only on $\|\varphi\|_\infty$, $\delta$ and $T$. Recall
\[E(s):=e^{\int_0^t \lambda(\gamma(r)) dr} [E_0(s)+\lambda(\gamma(s))u_1(s)].\]
By Lemma \ref{C.2},
$\dot E(s)=0$ a.e. on $[0,t]$. It follows that
\[E_0(s)\leq e^{\lambda T}(|E_0(s_0)|+\lambda_0 K)+\lambda_0 K:=F_1.\]
By (CON) we have
\begin{equation*}
\begin{aligned}
  L(\gamma(s),\frac{\dot\gamma(s)}{1+\|\dot\gamma(s)\|})-L(\gamma(s),\dot\gamma(s))&\geq (\frac{1}{1+\|\dot\gamma(s)\|}-1)\frac{\partial L}{\partial \dot x}(\gamma(s),\dot\gamma(s))\cdot \dot \gamma(s)
  \\ &\geq (\frac{1}{1+\|\dot\gamma(s)\|}-1)(F_1+L(\gamma(s),\dot\gamma(s))).
\end{aligned}
\end{equation*}
We denote by $K_3$ the bound of $L(x,\dot x)$ for $\|\dot x\|\leq 1$. Then we have
\[L(\gamma(s),\dot\gamma(s))\leq 2K_3+F_1.\]
By (SL), $\|\dot\gamma(s)\|$ is bounded by a constant depends only on $\|\varphi\|_\infty$, $\delta$ and $T$.

\bigskip

\noindent \textbf{Step 2. Lipschitz estimate of $(x,t)\mapsto T^-_t\varphi(x)$.} We first show that $u(x,t):=T^-_t\varphi(x)$ is locally Lipschitz in $x$. For any $r>0$ with $2r<\delta$, given $(x_0,t)\in M\times[\delta,T]$ and $x$, $x'\in B(x_0,r)$, denote by $d_0:=d(x,x')\leq 2r<\delta$ the Riemannian distance between $x$ and $x'$, we have
\begin{equation*}
\begin{aligned}
  u(x',t)-u(x,t)\leq& \int_{t-d_0}^{t}\bigg[L(\alpha(s),\dot \alpha(s))-\lambda(\alpha(s))u(\alpha(s),s)\bigg]ds
  \\ &-\int_{t-d_0}^{t}\bigg[L(\gamma(s),\dot{\gamma}(s))-\lambda(\gamma(s))u(\gamma(s),s)\bigg]ds,
\end{aligned}
\end{equation*}
where $\gamma(s)$ is a minimizer of $u(x,t)$ and $\alpha:[t-d_0,t]\rightarrow M$ is a geodesic satisfying $\alpha(t-d_0)=\gamma(t-d_0)$ and $\alpha(t)=x'$ with constant speed. By Step 1, the bound of $\|\dot \gamma(s)\|$ depends only on $\|\varphi\|_\infty$, $\delta$ and $T$. Since
\begin{equation*}
  \|\dot \alpha(s)\|\leq \frac{d(\gamma(t-d_0),x')}{d_0}\leq \frac{d(\gamma(t-d_0),x)}{d_0}+1,
\end{equation*}
and $d(\gamma(t-d_0),x)\leq \int_{t-d_0}^t\|\dot \gamma(s)\|ds$, the bound of $\|\dot \alpha(s)\|$ also depends only on $\|\varphi\|_\infty$, $\delta$ and $T$. Exchanging the role of $(x,t)$ and $(x',t)$, one obtain that $|u(x,t)-u(x',t)|\leq J_1d(x,x')$, where $J_1$ depends only on $\|\varphi\|_\infty$, $\delta$ and $T$. By the compactness of $M$, we conclude that for $t\in[\delta,T]$, the value function $u(\cdot,t)$ is Lipschitz on $M$.

We are now going to show the locally Lipschitz continuity of $u(x,t)$ in $t$. Given $t$ and $t'$ with $\delta\leq t<t'\leq T$. Let $\gamma:[0,t']\rightarrow M$ be a minimizer of $u(x,t')$, then
\begin{equation*}
\begin{aligned}
  u(x,t')-u(x,t)=u(\gamma(t),t)-u(x,t)+\int_t^{t'}\bigg[L(\gamma(s),\dot{\gamma}(s))-\lambda(\gamma(s))u(\gamma(s),s)\bigg]ds,
\end{aligned}
\end{equation*}
where the bound of $\|\dot \gamma(s)\|$ depends only on $\|\varphi\|_\infty$, $\delta$ and $T$. We have shown that for $t\geq \delta$, the following holds
\begin{equation*}
  u(\gamma(t),t)-u(x,t)\leq J_1d(\gamma(t),x)\leq J_1\int_t^{t'}\|\dot \gamma(s)\|ds\leq J_2(t'-t).
\end{equation*}
Thus, $u(x,t')-u(x,t)\leq J_3(t'-t)$, where $J_3$ depends only on $\|\varphi\|_\infty$, $\delta$ and $T$. The condition $t'<t$ is similar. We conclude the Lipschitz continuity of $u(x,\cdot)$ on $[\delta,T]$.
\end{proof}

Let $\|T^-_t\varphi(x)\|_\infty\leq K$ for all $t\geq 0$, with the bound $K$ independent of $t$. Note that $T^-_t\varphi(x)=T^-_1\circ T^-_{t-1}\varphi(x)$. {Fix $\delta=1/2$ and $T=1$ in Theorem \ref{A.1}. It follows that} the Lipschitz constant of $T^-_1\circ T^-_{t-1}\varphi(x)$ depends only on $K$, which is independent of $t$. This completes the proof of Proposition \ref{mthlip}.

%

\end{document}